\documentclass[12pt]{article}
\usepackage{latexsym}
\usepackage{amsfonts}
\usepackage{amsmath}
\usepackage[pdftex]{graphicx}
\usepackage[format=hang,font=sf,labelfont=bf]{caption}
\usepackage{pifont}
\usepackage{enumerate}
\usepackage{rawfonts}
\input{prepictex}
\input{pictex}
\input{postpictex}
\numberwithin{equation}{section}
\usepackage[OT2,OT1]{fontenc}
\def\cyr{%
\renewcommand\rmdefault{wncyr}%
\renewcommand\sfdefault{wncyss}%
\renewcommand\encodingdefault{OT2}%
\normalfont
\selectfont}
\DeclareMathAlphabet{\zap}{OT1}{pzc}{m}{it}
\DeclareTextFontCommand{\textcyr}{\cyr}
\newcommand{\rad}{\text{\cyr   ya}}
\newcommand{\Rad}{\text{\sc r}}
\newcommand{\dom}{\text{\cyr   D}}
\newcommand{\starry}{\text{\ding{73}}}
\def\uhs{\mathcal{H}^3}
\def\ZZ{\mathbb Z}
\def\RR{\mathbb R}
\def\CC{\mathbb C}
\def\HH{\mathbb H}

\def\TT{\mathbb{T}}

\DeclareMathAlphabet{\zap}{OT1}{pzc}{m}{it}
\def\CP{\mathbb{CP}}
\def\eea{\end{eqnarray*}}

\DeclareMathOperator{\dist}{dist}

\newtheorem{main}{Theorem}

\newtheorem{thm}{Theorem}[section]
\newtheorem{lem}[thm]{Lemma}
\newtheorem{prop}[thm]{Proposition}

\newtheorem{defn}[thm]{Definition}

\newenvironment{proof}{\medskip \noindent
{\bf Proof.}}{\hfill \rule{.5em}{1em}
\\}
\newenvironment{xpl}{\mbox{ }\\ {\bf  Example}\mbox{ }}{
\hfill $\diamondsuit$\mbox{}\bigskip}

\begin{document}
\sloppy

\title{Anti-Self-Dual $4$-Manifolds, Quasi-Fuchsian\\   Groups, and Almost-K\"ahler Geometry}

\author{Christopher J. Bishop\thanks{Supported 
in part by 
NSF Grant DMS-1608577.}
\quad and \quad Claude LeBrun\thanks{Supported 
in part by  NSF grant DMS-1510094.} 
\\ Department of Mathematics, 
Stony Brook University
  }

\date{}
\maketitle

\hspace{1.8in}
\begin{minipage}{2.4in}
\begin{quote} {\em For Karen Uhlenbeck, in celebration  of
her seventy-fifth birthday. }
\end{quote}
\end{minipage}

\bigskip 

\begin{abstract}
It is known that the  almost-K\"ahler anti-self-dual metrics on a given $4$-manifold sweep out   an open subset in the moduli space
of anti-self-dual metrics. However, we show here by example
that this subset is not  generally  closed, and so need not  sweep out entire
connected components in the moduli space. Our construction  hinges on an unexpected link between 
 harmonic functions on certain hyperbolic $3$-manifolds and self-dual harmonic $2$-forms
on  associated $4$-manifolds. 
\end{abstract} 

\section{Introduction} 
An oriented Riemannian $4$-manifold $(M^4,g)$ is said to be {\em anti-self-dual} if it satisfies $W_+=0$, where 
the self-dual Weyl  curvature $W_+$ is by definition the orthogonal projection  of the Riemann curvature tensor $\mathcal{R}\in \odot^2\Lambda^2$
into the  trace-free symmetric square $\odot^2_0\Lambda^+$ of the bundle of self-dual $2$-forms. 
This is a conformally invariant condition, and so is best understood as a condition on 
 the conformal
class $[g]:=\{ u^2 g ~|~ u: M \to \RR^+\}$ rather than on the representative metric $g$. 
For example, any oriented {\em locally-conformally-flat} $4$-manifold is anti-self-dual; indeed,  these are precisely the $4$-manifolds that are anti-self-dual
with respect to {\em both} orientations. One compelling reason for the study of anti-self-dual $4$-manifolds is that $W_+=0$ is exactly the integrability 
condition needed to make the total space of the $2$-sphere bundle $S(\Lambda^+)\to M$ into a complex $3$-fold $Z$ in a natural way \cite{AHS},
and  the {\em Penrose correspondence} \cite{pnlg} then allows one to completely reconstruct 
$(M^4,[g])$ from the complex geometry of the so-called {twistor space} $Z$.

A rather different  link between anti-self-dual metrics and complex geometry is provided by the observation that
a K\"ahler manifold $(M^4,g,J)$ of complex dimension two is anti-self-dual if and only if its scalar curvature vanishes \cite{lebicm}. 
This makes a special class of extremal K\"ahler manifolds  susceptible to study via twistor theory, and 
 led, in the early 1990s,  to various results on scalar-flat K\"ahler surfaces that anticipated  more recent 
theorems regarding more general  extremal K\"ahler manifolds.

Of course,  the K\"ahler condition is far from  conformally invariant. Rather, in the present context, it should be thought of as providing a 
preferred conformal gauge for those special conformal classes that can be represented by K\"ahler metrics. But
even among anti-self-dual conformal classes, those that can be represented by K\"ahler metrics tend to be 
highly non-generic. The following example  nicely illustrates this phenomenon. 

\begin{xpl}
Let $\Sigma$ be a compact Riemann surface of genus ${\zap g}\geq 2$, and let $M$ denote the compact oriented $4$-manifold 
$\Sigma \times \CP_1$. If we equip $\Sigma$ with its hyperbolic metric of Gauss curvature $K=-1$ and equip $S^2=\CP_1$
with its usual ``round'' metric of Gauss curvature $K=+1$, the product metric on $M= \Sigma \times S^2$ is scalar-flat
K\"ahler, and hence anti-self-dual. Moreover, since this metric admits orientation-reversing isometries, it is actually locally 
conformally flat. This last fact illustrates a useful consequence \cite{AHS} of the $4$-dimensional signature formula:
if $M$ is a compact oriented $4$-manifold (without boundary) that has  signature $\tau (M)=0$, then every anti-self-dual 
metric on $M$ is locally conformally flat. 

Now the scalar-flat K\"ahler metric we have just described on $M=\Sigma \times S^2$ has universal cover $\mathcal{H}^2\times S^2$,
where $\mathcal{H}^2$ denotes the hyperbolic plane. In fact,  one can show \cite{burbar,lsd} that {\em every} scalar-flat K\"ahler metric
on $\Sigma \times S^2$  has universal cover homothetically isometric to this fixed model. Thus,
starting  with the representation $\pi_1(M)= \pi_1(\Sigma ) \hookrightarrow PSL(2,\RR) = SO_+(2,1)$ 
that uniformizes $\Sigma$, 
 any deformation of our example through 
 scalar-flat K\"ahler metrics  of fixed total volume arises  from a deformation though  representations $\pi_1(\Sigma ) \hookrightarrow SO_+(2,1)\times SO(3)$.  
On the other hand,  $\mathcal{H}^2\times S^2$ is conformally isometric to the complement $S^4-S^1$ of an equatorial circle in the 
$4$-sphere, so the general deformation of the locally conformally flat structure  on $M=\Sigma \times S^2$  instead just corresponds to a  family of 
homomorphisms  $\pi_1(\Sigma ) \hookrightarrow SO_+(5,1)$ taking values in  the group of M\"obius transformations
of the $4$-sphere. Remembering to only  count these representations up to conjugation, we see  that the moduli space of 
anti-self-dual conformal classes on $M$ has dimension $30 ({\zap g}-1)$, but that only a subspace of  dimension $12 ({\zap g}-1)$ 
arises from conformal structures that contain scalar-flat K\"ahler metrics. 
\end{xpl}

Because of this, we can hardly expect scalar-flat-K\"ahler metrics to provide a reliable model for general 
anti-self-dual metrics, even  on $4$-manifolds that arise as compact complex surfaces. However, 
the larger class of {\em almost-K\"ahler} anti-self-dual metrics shares
many of the remarkable properties of the scalar-flat K\"ahler metrics, and yet sweeps out an {\sf open} region in the moduli space of anti-self-dual conformal structures. 

Recall that  an oriented  Riemannian manifold $(M,g)$ equipped with a closed 
$2$-form $\omega$ 
is said to be {\em almost-K\"ahler} if there is  an orientation-compatible almost-complex structure $J: TM\to TM$, $J^2=-\mathbf{1}$, such that 
$g=\omega (\cdot , J\cdot )$; in this language,  a K\"ahler manifold  just becomes an almost-K\"ahler manifold for which the 
almost-complex structure $J$ happens to be  integrable.  But because $J$ is algebraically determined by $g$ and $\omega$, there is an equivalent  reformulation that avoids  mentioning $J$ explicitly. Indeed, 
an oriented Riemannian $2m$-manifold $(M,g)$ is almost-K\"ahler with respect to the closed $2$-form $\omega$ iff 
    $\ast \omega = \omega^{m-1}/(m-1)!$ and $|\omega|= \sqrt{m}$, where the Hodge star and pointwise norm on $2$-forms are those determined by  $g$ and the 
    fixed orientation. 
In particular, we see that  $\omega$ is a  {\sf harmonic}  $2$-form on $(M^{2m},g)$, and that  $\omega$ is an orientation-compatible symplectic form
on $M$. 

But this also makes it  clear   that the $4$-dimensional case is  transparently  simple  and  natural. Indeed, 
a compact oriented Riemannian $4$-manifold $(M,g)$ is almost-K\"ahler with respect to $\omega$ iff $\omega$ is a {\sf self-dual harmonic $2$-form}
on $(M^4,g)$ of {\sf constant length }$\sqrt{2}$. However, since the Hodge  star  operator is conformally invariant on middle-dimensional forms,  
a $2$-form on a $4$-manifold is harmonic with respect to  $g$ iff it is harmonic with respect to  every other metric in the conformal class $[g]$. Since 
a conformal change of metric $g\rightsquigarrow u^2g$ changes the point-wise norm of a $2$-form by $|\omega|\rightsquigarrow u^{-2} |\omega|$, 
this means that any  harmonic self-dual form $\omega$ that is simply {\sf everywhere non-zero} determines
a unique $\hat{g}= u^2g\in [g]$ such that  $(M,\hat{g})$  is almost-K\"ahler with respect to $\omega$. Moreover, 
the dimension 
$b_+(M) = [b_2(M) + \tau(M)]/2$ of the space of self-dual harmonic  $2$-forms is a topological invariant of $M$, and 
the collection of these forms depends continuously on the space of Riemannian metrics (with respect, for example, to the $C^{1,\alpha}$ topology). 
Thus, if there is a nowhere-zero self-dual harmonic $2$-form $\omega$ with respect to $g$, the same is true for every metric 
$\tilde{g}$  that is sufficiently close to $g$ in the $C^2$ topology. It therefore follows that  the almost-K\"ahler metrics sweep out an open subset of
the  space of conformal classes on any smooth compact 
$4$-manifold $M^4$. We emphasize that this is a strictly $4$-dimensional phenomenon; it certainly does not persist in higher dimensions.

Because of this,  almost-K\"ahler anti-self-dual metrics can provide   an interesting window into the 
world of general anti-self-dual metrics, at least   on $4$-manifolds that happen to admit symplectic structures.
Since  the anti-self-duality condition $W_+=0$ largely compensates for the  extra freedom in the curvature tensor
that would otherwise result from relaxing the K\"ahler condition,  many features familiar from the scalar-flat K\"ahler case
turn out to  persist in this broader context. 
One  particularly intriguing consequence is that it is 
  not difficult to find obstructions to the existence of  almost-K\"ahler anti-self-dual metrics, even though the known obstructions to the existence of general anti-self-dual metrics are few and far between.  For example, while we know \cite{klp,rolsing} that there are  scalar-flat K\"ahler metrics (and hence anti-self-dual metrics) 
 on $\CP_2 \# k \overline{\CP}_2$ for $k\geq 10$, we also know \cite{inyoungagag,lebdelpezzo} that  almost-K\"ahler anti-self-dual metrics definitely {\sf do not} exist on such blow-ups
of the complex projective plane  at $k\leq 9$ points. Is the latter indicative of a deeper non-existence theorem for 
 general anti-self-dual metrics? Or is   this merely the sort of false hope that  arises  from staring too long at a  mirage? 
 
 This paper will try to shed some light on these  matters by  investigating a related question. 
  If a smooth compact $4$-manifold admits an almost-K\"ahler anti-self-dual metric, is every anti-self-dual metric
 in the same component of the moduli space also almost-K\"ahler? Since we have seen that the almost-K\"ahler condition is open on the level of conformal classes, 
 this question amounts asking whether it is also {\sf closed} in the  anti-self-dual context. Our results will show  that the answer is  {\sf no}. Indeed, we will display 
 a large family of counter-examples that arise  from the theory of quasi-Fuchsian groups.

\begin{main}
\label{main1}
 There is an  integer $N$ such that, whenever 
 $\Sigma$ is a compact oriented  surface of even genus ${\zap g}\geq N$,  the $4$-manifold $M=\Sigma \times S^2$ admits
 locally-conformally-flat conformal classes $[g]$ that  cannot be represented by  almost-K\"ahler metrics. 
Moreover, certain   such $[g]$  arise from quasi-Fuchsian groups $\pi_1(\Sigma ) \hookrightarrow SO_+(3,1)\subset SO_+(5,1)$, and so 
can be exhibited  as locally-conformally-flat deformations of the conformal structures represented by  scalar-flat K\"ahler product metrics on $\Sigma \times S^2$. 
\end{main}

While the examples  described by this result are all locally conformally flat, and thus live on $4$-manifolds of signature zero, a variant of the same construction produces many 
explicit examples that live on $4$-manifolds with $\tau < 0$, and so are certainly {\sf not} locally conformally flat. 
These arise in connection with  the second author's explicit construction \cite{leblown} of scalar-flat K\"ahler metrics on 
blown-up ruled surfaces. The main idea is to deform  the hyperbolic $3$-manifolds that played a central role in the earlier
construction, by replacing Fuchsian with quasi-Fuchsian subgroups of $PSL(2, \CC)$.

\begin{main}  
\label{main2}
Let $k\geq 2$ be an integer, and 
let $N$ be the integer of Theorem \ref{main1}. 
Then if $\Sigma$ is a compact oriented  surface of even genus ${\zap g}\geq N$, 
 the connected sum $M= (\Sigma \times S^2) \# k \overline{\CP}_2$ admits anti-self-dual conformal structures
 that  cannot be represented by  almost-K\"ahler metrics. Moreover, some such $[g]$ can be explicitly constructed from 
 configurations of $k$ points in  quasi-Fuchsian hyperbolic $3$-manifolds diffeomorphic to   $\Sigma \times \RR$, and 
so   can be exhibited  as anti-self-dual deformations  of  conformal structures that are represented by  scalar-flat K\"ahler structures on blown-up ruled surfaces. 
\end{main}

\bigskip 

\noindent {\bf Dedication and acknowledgments:} This paper is dedicated to Karen Uhlenbeck, on the occasion of her $75^{\rm th}$ birthday. 
While Karen is of course primarily known for her deep  contributions to geometric analysis, her role in 
encouraging   the work of younger mathematicians has also been a consistent feature of her long and fruitful career. 
The second author, whose taste and interests have been significantly shaped and influenced  by Uhlenbeck's work, 
would therefore like to take this opportunity  to thank Karen for the encouragement she offered him at the outset
of this project. On the other hand, the  first author, who had  presumably been off Prof.\ Uhlenbeck's radar for many years, would  like to 
belatedly thank her  for having passed him on his  
topology oral exam at the University of Chicago in 1984, while offering  this paper as evidence  
that this decision might not have been a mistake after all. Finally, the authors would  like to 
thank many other  colleagues, including   Dennis Sullivan, Ian Agol,  and  Inyoung Kim,   for 
helpful and stimulating conversations, as well as  the  anonymous referees for  carefully reading of the paper and 
suggesting  some  useful clarifications of our 
  exposition.

\section{Fuchsian and Quasi-Fuchsian Groups}

Let $\mathcal{H}^3$ denote hyperbolic $3$-space, which we will visualize using either  the Poincar\'e ball model or  the upper-half-space model. 
Either way, we see a $2$-sphere at infinity; in the Poincar\'e model, it is simply the boundary $2$-sphere of the closed $3$-ball $D^3$, 
while in the upper-half-space model it 
becomes the boundary plane plus an extra point,  called $\infty$.  In either picture,  isometries of $\mathcal{H}^3$ extend to 
to the boundary sphere as conformal transformations, and  every global conformal transformation of $S^2$ conversely arises  this way from a 
unique  isometry. We will henceforth choose to emphasize isometries of $\mathcal{H}^3$ and conformal maps of $S^2=\CP_1$ that {\sf preserve orientation}. 
The group of such isometries is then exactly the  complex automorphism group ${PSL}(2, \CC)$ of $\CP_1$; the fact that this 
 can be identified with the Lorentz group ${SO}_+(3,1)$ (which most naturally acts on yet a third model for $\mathcal{H}^3$,  the hyperboloid of unit future-pointing 
 time-like vectors in Minkowski space)  is
one of those rare low-dimensional   coincidences in Lie group theory  that  underlie many important phenomena in   low-dimensional geometry and topology. 
Another relevant coincidence is that  ${SO}_+(5,1)= PGL(2,\HH)$, so that    oriented conformal transformations of the $4$-sphere can be understood as 
fractional linear transformation of the quaternionic projective line $\mathbb{HP}_1$; thus the natural exension  $SO_+(3,1)\hookrightarrow SO_+(5,1)$ 
of conformal transformations from $S^2$ 
to $S^4$  can also be understood as arising from the  inclusion $PSL(2, \CC)\hookrightarrow PGL(2,\HH)$ induced by including the 
complex numbers $\CC$ into the quaternions $\HH$.

A {\em Kleinian group} $\varGamma$ is by definition \cite{maskit} a discrete subgroup of  ${PSL}(2, \CC)$. Since {\sf discrete} means that the identity element is 
isolated,  this implies that the orbit of any point in  $\mathcal{H}^3\subset D^3$ can only accumulate on the boundary sphere. 
The set  of accumulation points of any orbit 
 is called the {\sf limit set}, and  denoted $\Lambda  = \Lambda (\varGamma)$;  this can easily be shown to be  independent of the particular orbit we choose.
A {\em Fuchsian group}  is by definition a Kleinian group which sends some geometric disk $D^2\subset S^2$ to itself; and since the boundary circle of any such disk is
 the image of $\mathbb{RP}^1 \subset \CP_1$ under a M\"obius transformation, 
 this is equivalent to saying that a Fuchsian group is a Kleinian group which is conjugate to a subgroup of
 ${PSL}(2, \RR)$.  
   In this case, the limit set $\Lambda (\varGamma)$ must be a closed subset of the invariant circle. A Fuchsian group is said to be of the {\sf first type} if $\Lambda (\varGamma)$  is the whole circle. 
   A Kleinian group $\varGamma$ is called {\em quasi-Fuchsian} if it is {\sf quasiconformally conjugate} to a Fuchsian group $\varGamma^\prime$ of the first type, meaning that 
   $\varGamma = \Phi^{-1}  \circ \varGamma^\prime  \circ \Phi$  for some quasiconformal homeomorphism $\Phi$ of $\CP_1$. In particular, the limit set of such a group is a 
   {\sf quasi-circle}, meaning  a 
   Jordan curve that is the image of a geometric circle under some quasiconformal map. 
    This implies \cite{gehringvais} 
   that the limit set has Hausdorff dimension $<2$, and so in particular has 
   Lebesgue area zero. We note in passing  that there are other  equivalent characterizations of such groups; for example,  
   a finitely generated Kleinian group is quasi-Fuchsian if and only if its limit set is a closed Jordan curve.

\begin{figure}[htb]
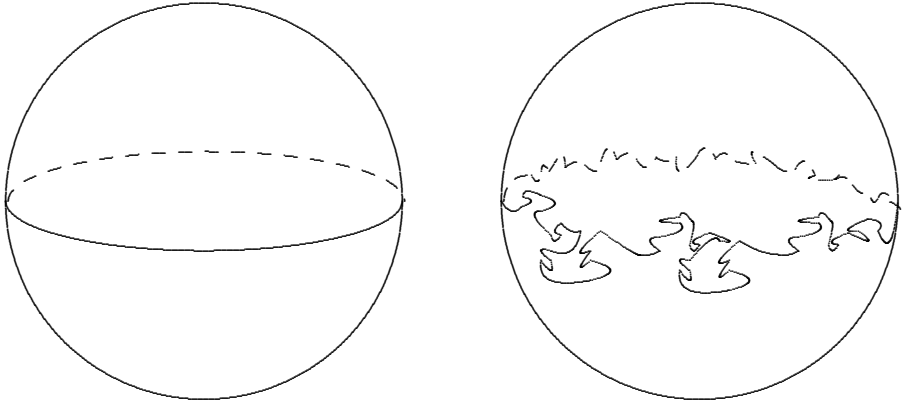

\centerline{
\beginpicture
\setplotarea x from -50 to 400, y from 40 to 180 
\startrotation by 0.75 0 about 0 0
\circulararc -360 degrees from 0 150  center at 100 150 
\circulararc -360 degrees from 250 150  center at 350 150 
\ellipticalarc axes ratio 4:1 180 degrees from 0 150 
center at 100 150 
{\setquadratic 
\plot  250 150   255 145   260 147 /
\plot 260 147   265 150   261 152    /
\plot 261 152   263 155 270 153  /
\plot 270 153  277 149  270 145 /
\plot 270 145  267 142  275 133 /
\plot 275 133 277 130  280 133 /
\plot  280 133 277 135 280 139  /
\plot   280 139  285 135 290 135 /
\plot   290 135  285 125 275 125 /
\plot   275 125   270 125  275 120 /
\plot   275 120  275 115 270 120 /
\plot   270 120   275 110  295 110 /
\plot    295 110 305 115 295 120 /
\plot  295 120 290 117 295 125 /
\plot   295 125 290 125 300  135 /
\plot   300  135 315 125 320 123 /
\plot  320 123   325 123   330 125 /
\plot  330 125   325 130   330 132 /
\plot 330 132   340 135   340 139    /
\plot 340 139   330 140 340 142   /
\plot 340 142  342 144  346 140 /
\plot 346 140  347 137  345 128 /
\plot 345 128 347 125  350 128 /
\plot  350 128 347 130 355 134  /
\plot   355 134  365 130 360 130 /
\plot   360 130 355 130 345 120 /
\plot  345 120  340 120  345  115 /
\plot   345 115  345 110 340 115 /
\plot   340 115   345 105  365 105 /
\plot    365 105 375 110 365 115 /
\plot  365 115 360 112 365 120  /
\plot   365 120  360 120 370  130 /
\plot   370  130 380 125  390 123 /
\plot  390 123   395 123   400 125 /
\plot  400 125   395 130   400 132 /
\plot 400 132   410 135   410 139    /
\plot 410 139   400 140 410 142   /
\plot 410 142  412 144  416 140 /
\plot 416 140  417 137  415 128 /
\plot 415 128 417 125  420 128 /
\plot  420 128 417 130 425 134  /
\plot   425 134  435 130 430 135 /
\plot   430 135 425 136 435 141 /
\plot   435 141 441 139 440 135 /
\plot   440 135 446 130 450 150 /
}
{\setdashes 
\ellipticalarc axes ratio 4:1 -180 degrees from -4 150
center at 96 150
\setquadratic 
\plot  246 150   250 160   260 163 /
\plot  260 163   263 160   265 165 /
\plot  265 165   266 168   269 165 /
\plot  269 165   275 173   277 167 /
\plot   277 167  283 170 287 168 /
\plot   287 168  293 165   295 170 /
\plot    295 170  299 177   303 171 /
\plot    303 171  308 174   313 172 /
\plot    313 171  318 167   323 172 /
\plot   323 172  328 170   333 171 /
\plot   333 172  335 165  343 172 /
\plot   343 172  348 177   353 171 /
\plot   353 172  358 174  363 170 /
\plot   363 170  368 169   373 169 /
\plot   373 169  372 175   383 168 /
\plot   383 168  388 165   393 166 /
\plot   393 166  398 170  403 163 /
\plot   403 163  401 161   415 161  /
\plot   415 163  414 162   413 161  /
\plot   413 161  418 162   422 158    /
\plot   422 158  427 154  430 158 /
\plot   435 152  432 150   430 158 /
\plot   440 153  437 154   435 152 /
\plot   446 146  444 148   440 150 /
}
\endpicture
}
\caption{\label{limitsets}
While the limit set of a type-one Fuchsian group is a geometric  circle, for a strictly  quasi-Fuchsian group it is instead 
a quasi-circle that is a self-similar  Jordan curve  of  Hausdorff dimension $>1$.}
\end{figure}

   The special class of quasi-Fuchsian groups $\varGamma$ that will concern us here consists of  those   $\varGamma$  that are 
   group-isomorphic to the fundamental group $\pi_1(\Sigma )$
   of some
    compact oriented surface $\Sigma$ of genus ${\zap g}\geq 2$. We will call these\footnote{Since the standard terminology would instead describe such $\varGamma$  as 
    {\sf finitely-generated  convex-co-compact  
    quasi-Fuchsian groups without  elliptic elements}, the introduction of a shorter name seems both necessary and appropriate! } 
     quasi-Fuchsian groups of {\em Bers type}, in honor of Lipman Bers' pioneering contribution
   to the subject. Given    two orientation-compatible complex structures ${\zap j}$ and ${\zap j}^\prime$ on $\Sigma$, Bers \cite{bersim} 
   showed that there is a  quasi-Fuchsian group 
   $\pi_1(\Sigma) \hookrightarrow PSL(2, \CC)$ with  limit set a quasi-circle $\Lambda$,  
   such that $(\CP_1-\Lambda) /\pi_1(\Sigma) $ is biholomorphic to the  disjoint union $(\Sigma, {\zap j}) \sqcup (\Sigma, -{\zap j}^\prime)$. 
   The hyperbolic manifold $X=\mathcal{H}^3/\pi_1(\Sigma )$ is then diffeomorphic 
   to $\Sigma \times (-1,1)$, and the  $3$-manifold-with-boundary $\overline{X}:=(D^3 - \Lambda)/\pi_1(\Sigma)$, which is diffeomorphic to $\Sigma \times [-1,1]$,
   carries a conformal structure which extends the conformal class  of the hyperbolic  metric  on  $X\subset \bar{X}$ and
    induces the two specified conformal structures on the two boundary components $\Sigma \times \{ \pm 1\}$.   The quasi-Fuchsian group that accomplishes this
    is moreover unique up to conjugation in $PSL(2, \CC)$, so that these Bers groups are classified by pairs of points in Teichm\"uller space.

 \begin{figure}[htb]
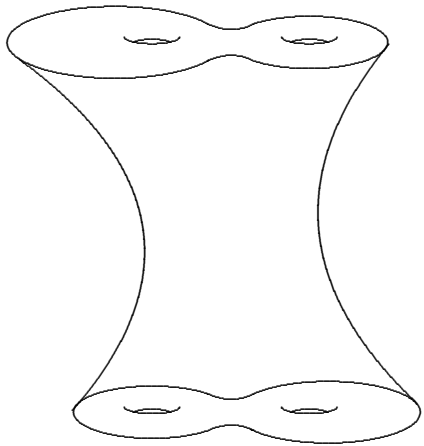

\centerline{
 \beginpicture
\startrotation by 0.7 0 about 150 100
\setplotarea x from 0 to 300, y from 40 to 200
\ellipticalarc axes ratio 3:1  300 degrees from 150 240
center at 100 230
\ellipticalarc axes ratio 3:1  -270 degrees from 175 240
center at 205 230
\ellipticalarc axes ratio 4:1 -180 degrees from 135 233
center at 120 233
\ellipticalarc axes ratio 4:1 145 degrees from 130 230
center at 120 229
\ellipticalarc axes ratio 4:1 180 degrees from 190 233
center at 205 233
\ellipticalarc axes ratio 4:1 -145 degrees from 195 230
center at 205 229
\ellipticalarc axes ratio 3:1  270 degrees from 150 40
center at 120 30
\ellipticalarc axes ratio 3:1  -290 degrees from 175 40
center at 215 30
\ellipticalarc axes ratio 4:1 -180 degrees from 135 33
center at 120 33
\ellipticalarc axes ratio 4:1 145 degrees from 130 30
center at 120 29
\ellipticalarc axes ratio 4:1 180 degrees from 190 33
center at 205 33
\ellipticalarc axes ratio 4:1 -145 degrees from 195 30
center at 205 29
{\setquadratic 
\plot 150 40    163 37    175 40   /
\plot 150 20    163 23    175 20  /
\plot 150 240    163 237    175 240   /
\plot 150 220    163 223    175 220  /
\plot 263 35    210 130   248 229   /
\plot 77 31    115 130   48 221   /
}
\endpicture
}
\caption{\label{quasimodo}
If $\varGamma$ is a quasi-Fuchsian group 
  of Bers type, the associated hyperbolic $3$-manifold $\mathcal{H}^3/\varGamma$  
  has two ends, and is characterized by the  conformal structures it induces on the   two associated 
surfaces at infinity. In the Fuchsian case, the  two conformal structures at infinity  coincide.}
\end{figure}

\section{Constructing Anti-Self-Dual    $4$-Manifolds}

We will now construct a menagerie of  explicit anti-self-dual $4$-manifolds, starting  from any quasi-Fuchsian group $\varGamma\subset PSL (2, \CC)$ of Bers type. 
Recall that our definition requires $\varGamma$ to 
have 
 limit set  $\Lambda (\varGamma) \subset \CP_1$ equal to  a quasi-circle, and to be group-isomorphic to 
$\pi_1(\Sigma )$    for some  compact oriented surface $\Sigma$ of  genus ${\zap g}\geq 2$.
 Let  $X=\mathcal{H}^3/\varGamma$ denote the associated 
hyperbolic $3$-manifold, and let $\overline{X}:=(D^3 - \Lambda)/\varGamma$ be its canonical compactification as   a $3$-manifold-with-boundary. 
We will use $h$ to denote the hyperbolic metric on $X$, and will let $[\bar{h}]$ denote the conformal structure on $\overline{X}$ induced
by the Euclidean conformal structure on $D^3$, all the while remembering that the restriction of  $[\bar{h}]$ to the interior
$X$ of $\overline{X}$ is just  the conformal class $[h]$ of the hyperbolic metric $h$.

    The simplest version of our construction proceeds by defining $P$ to be the $4$-manifold $X\times S^1$, and  then compactifying this  as  
    $M=(\overline{X}\times S^1)/\negthickspace\sim$, where the equivalence relation $\sim$ collapses $(\partial \overline{X}) \times S^1$ to $\partial \overline{X}$ by contracting each
    circle factor to a point. As a set, $M$ is therefore just the disjoint union of $P$ and  $\partial \overline{X}$.  The point of interest, though, 
    is that the topological space $M$ can be made into a smooth $4$-manifold in a way that simultaneously endows it with a     locally-conformally-flat 
    conformal 
   structure. To see this, let us first  recall that the Riemannian product ${\mathcal H}^3 \times S^1$ is conformally flat, since 
    \begin{equation}
\label{model}
z^2\left( \frac{dx^2+dy^2+dz^2}{z^2} + dt^2\right) =  (dx^2+ dy^2)+ (dz^2 + z^2 dt^2)
\end{equation}
    displays a certain conformal rescaling of   its metric as the Euclidean metric on $\RR^4-\RR^2 = \RR^2 \times (\RR^2 -\{ 0\})$, 
    written in cylindrical coordinates. 
   We can generalize this by letting  $${\zap u}: \overline{X}\to [0,\infty )$$ be a smooth defining function for $\partial \overline{X}$, in the sense that  $\partial \overline{X}= {\zap u}^{-1}(\{ 0\})$
   and that $d{\zap u}\neq 0$ along $\partial \overline{X}$. Inspection of \eqref{model}  then demonstrates that 
   \begin{equation}
\label{ansatz0}
g={\zap u}^2(h+dt^2)
\end{equation}
defines a smooth locally-conformally-flat
   metric on $M$, since near any boundary point of $\overline{X}$ we can choose local coordinates $(x,y,z)$ that express $h$  in the upper-half-space model, 
   and  we then automatically have  ${\zap u}={\zap w}z$ for some smooth positive function ${\zap w}$ on the corresponding coordinate domain. Of course, the metric 
   $g$ defined by \eqref{ansatz0} depends on the defining function ${\zap u}$, but its conformal class $[g]$ does not. This gives 
    $M=(\overline{X}\times S^1)/\negthickspace\sim$ the structure of a smooth oriented locally-conformally-flat $4$-manifold in a very natural manner. It follows  that 
 $M$ is  diffeomorphic to $\Sigma \times S^2$, since  $\overline{X}$ is diffeomorphic to $\Sigma \times [-1,1]$, and  collapsing each boundary
 circle of  $[-1,1]\times S^1$  exactly produces a $2$-sphere $S^2$.

    Of course, none of this should come as a  surprise, since the inclusion 
    $$\begin{array}{ccc}PSL (2,\CC) & \hookrightarrow & PGL(2,\HH ) \\ \| &  & \|  \\ SO_+(3,1) & \hookrightarrow &  SO_+(5,1)\end{array}$$
   exactly allows $\varGamma$ to act on $S^4= \mathbb{HP}_1$ in a manner that is free and properly discontinuous outside 
   the limit set  $\Lambda = \Lambda (\varGamma) \subset S^2\subset S^4$. 
  The smooth compact $4$-manifold $M$ we constructed above is exactly $(S^4-\Lambda)/\varGamma$, and the locally-conformally-flat conformal structure $[g]$ 
  with which we 
  endowed it  is simply the push-forward of the  standard conformal structure on $S^4$. However, the approach we have just detailed 
  has certain specific  virtues; not only does it 
      nicely  generalize to yield a construction of more general  anti-self-dual $4$-manifolds, but it will also lead, in  \S \ref{harmandammer} below,  to  a concrete picture of   
    harmonic $2$-forms on $(M,[g])$ in terms  of harmonic functions on $(X,h)$. 
    
    Before  proceeding further,   we should notice some other key features of the metrics $g$ defined by \eqref{ansatz0}. The vector field $\xi = \partial/\partial t$ is
    a {\em Killing field} of $g$, and generates an isometric action of $S^1=U(1)$ on $(M,g)$. We can usefully restate this by  observing that  $\xi$ is a {\em conformal Killing field} of
    $(M,[g])$, and  that the special metrics $g\in [g]$ given by  \eqref{ansatz0} are simply the metrics in the fixed conformal class
     that  are invariant under the induced circle action. 
    Now observe that $\overline{X}=M/S^1$, and that the inverse image  $P$ of $X\subset \overline{X}$  is a flat principle circle
    bundle --- namely, the trivial one!  We now mildly generalize the construction by   instead  considering {\sf arbitrary} {flat} principal $S^1$-bundles $(P,\theta)$ over $X$. 
     On such a principal bundle,  there is  still a  vector field $\xi$ that generates the free $S^1$-action on $P$, and  saying that $\theta$ is a connection $1$-form
      just means that it's  an $S^1$-invariant $1$-form on $P$ such that $\theta(\xi) \equiv 1$; requiring  that such a  connection be  {\sf flat} then imposes the condition 
     that $d\theta=0$, and, since this then says  that $\theta$ is locally exact,    is  obviously  equivalent to saying that there is a system of local trivializations of $P$ in which 
     $\theta = dt$ and $\xi=\partial/\partial t$. Given any smooth defining function ${\zap u}: \overline{X}\to [0,\infty )$ for $\partial \overline{X}$, 
       the  local arguments we used before now show that  we can compactify  
    $(P, {\zap u}^2[h+\theta^2])$ as a smooth  locally-conformally-flat $4$-manifold $(M,g)$ by adding a copy of $\partial\overline{X}$. 
    On the other hand,  the extra freedom of choosing a flat $S^1$-connection on $X$ is completely encoded in a monodromy homomorphism 
    $\pi_1(X) \to S^1 = SO(2)$, and since $\pi_1(X)=\pi_1(\Sigma ) \cong \varGamma$, the graph of such a homomorphism  is 
     a subgroup $\widetilde{\varGamma}\subset SO_+(3,1)\times SO(2)$
    that projects bijectively to the quasi-Fuchsian group $\varGamma\subset SO_+(3,1)$. Identifying $\widetilde{\varGamma}$ with its image under the natural inclusion 
     $SO_+(3,1)\times SO(2)\hookrightarrow SO_+(5,1)$ then allows  $\widetilde{\varGamma}$ to act conformally on $S^4$ with the same limit set $\Lambda\subset S^2\subset S^4$ as $\varGamma$,      
     and the additional locally-conformally-flat structures on $M\approx \Sigma \times S^2$  introduced above are   exactly the ones that  then arise as  quotients 
     $(S^4-\Lambda) /\widetilde{\varGamma}$. 
     For all of these conformal structures,  
      $(M,[g])$  comes equipped with an $S^1$-action generated by a conformal Killing field $\xi$; moreover, they all have $\overline{X}=M/S^1$,
    with  $\partial \overline{X}$ exactly given by the image of the zero locus of $\xi$. 

  We will now describe a generalization of  the above ansatz that constructs anti-self-dual $4$-manifolds that are {\sf not} locally conformally flat. 
  Let $(X,h)$ once again be the hyperbolic $3$-manifold associated with some quasi-Fuchsian group $\varGamma\subset PSL (2, \CC )$ of 
  Bers type;  and let us emphasize that we   take $\mathcal{H}^3$ to have a standard orientation, so that $X=\mathcal{H}^3/\varGamma$ also comes equipped with 
  a preferred  orientation from the outset.  For some positive integer $k$, now choose a configuration $\{ p_1 , p_2, \ldots , p_k\}$ of $k$ distinct points in $X$. Our objective
  will be to construct a compact anti-self-dual $4$-manifold $(M,[g])$ with a 
  semi-free\footnote{An  action is called {\sf semi-free} if it is free on the complement of its fixed-point set.}  conformally isometric $S^1$-action with  
  $k$ isolated fixed points $\{ \hat{p}_1, \hat{p}_2  \ldots , \hat{p}_k\}$   and two fixed surfaces $\Sigma_+$ and $\Sigma_-$, 
  such that 
  $M/S^{1}= \overline{X}$, with $\hat{p}_j$ mapping to $p_j$, and with the $\Sigma_+ \sqcup \Sigma_-$ mapping to $\partial \overline{X}$. 
   The construction  will actually require  the configuration $\{ p_1 , p_2, \ldots , p_k\}$  to satisfy a  mild constraint, but 
   configurations  with this property will  turn out to exist for all $k\geq 2$.   
  
  The  ingredients  needed for our  construction will include the Green's functions $G_{p_j}$ of the chosen points. By definition, 
  each $G_{p_j}: X-\{ p_j\}\to \RR^+$ is a positive harmonic function 
  that  tends to zero at $\partial \overline{X}$, and solves 
  \begin{equation}
\label{greens}
\Delta G_{p_j} = 2\pi \delta_{p_j}
\end{equation}
  in the distributional sense, where $\Delta = -\mbox{div grad}$ is the (modern  geometer's) Laplace-Beltrami operator  
 of   the hyperbolic metric $h$. 
  This Green's function can be constructed  explicitly by lifting the problem to 
  $\mathcal{H}^3$, where the inverse image of $p_j$ becomes the orbit  $\varGamma q_j$ of an arbitrarily chosen point  $q_j$ in the preimage. Superimposing 
  the hyperbolic Green's functions for the points in the orbit then leads one to express  the solution as a Poincar\'e  series 
  \begin{equation}
\label{serious}
G_{p_j}(q) = \sum_{\boldsymbol{\phi} \in \varGamma} \frac{1}{e^{2\dist (\boldsymbol{\phi} q_j,q)}-1},
\end{equation}
  where $\dist$ denotes the hyperbolic distance in $\mathcal{H}^3$. 
  The fact that this expression converges away from the orbit $\varGamma q_j$ 
 follows from  Sullivan's theorem on critical exponents \cite{sullivan} for Poincar\'e series, because the limit set $\Lambda (\varGamma )$ has Hausdorff dimension $< 2$. 
 It is then easy to show that the singular function defined by \eqref{serious} 
 solves \eqref{greens} in the distributional sense, and elliptic regularity therefore shows that 
 $G_{p_j}$ is smooth on $X-\{ p_j\}$. Moreover, Sullivan's theorem also implies that 
  $G_{p_j}$ extends continuously to the boundary of $\overline{X}$ by zero. 
 Regularity theory for  boundary-degenerate elliptic operators   \cite[Theorem 11.7]{graham} then implies that this  extension of $G_{p_j}$ 
 is actually smooth on  $\overline{X}-\{ p_j\}$, and has vanishing  normal derivative at $\partial \overline{X}$. 
 
 Let us next define a  harmonic function $V: X- \{ p_1, p_2, \ldots , p_k\} \to \RR^+$ by 
 \begin{equation}
\label{potential}
 V = 1+ G_{p_1} + G_{p_2} + \cdots + G_{p_k}.
\end{equation}
Since $V$ satisfies Laplace's equation on the complement of $\{ p_1, \ldots , p_k\}$, we have  $d\star dV=0$ in this region, 
and the $2$-form defined there by 
$$ F = \star dV$$
is therefore closed. Our construction now asks us to  find a principal circle bundle $P\to X- \{ p_1, p_2, \ldots , p_k\}$ equipped 
with a connection form $\theta$ whose curvature is exactly $F$. On any contractible region $U\subset  X- \{ p_1, p_2, \ldots , p_k\}$, this can always be done simply by taking
any $1$-form $\vartheta$ with $d\vartheta=F$, and then setting $\theta = dt+ \vartheta$ on $U \times S^1$. However, there is a 
cohomological obstruction to gluing these local models together consistently; namely, we need $[\frac{1}{2\pi} F]$ to be an integer class
in  deRham cohomology, because it will ultimately represent  the first Chern class
$c_1(P)\in H^2(X- \{ p_1, p_2, \ldots , p_k\}, \ZZ)$. This motivates the following definition:

\begin{defn} 
\label{intent}
If $X= \mathcal{H}^3/\varGamma$ is a quasi-Fuchsian hyperbolic $3$-manifold of Bers type, and if $\{ p_1, \ldots , p_k\}$ is a configurations  of  
$k\geq 0$ distinct points in $X$, we will say that $\{ p_1, \ldots , p_k\}$  is {\em quantizable} if  $\frac{1}{2\pi} \star dV $ represents an element of 
$H^2(X- \{ p_1, \ldots , p_k\}, \ZZ )\subset H^2 (X- \{ p_1, \ldots , p_k\}, \RR )$ in deRham cohomology. 
\end{defn}

 This ``quantization condition'' is equivalent to demanding that $\frac{1}{2\pi}\int_YF$ be  an integer for every  
smooth compact oriented surface $Y\subset  X- \{ p_1, p_2, \ldots , p_k\}$ without boundary. However,  $H_2(X- \{ p_1, p_2, \ldots , p_k\}, \ZZ)$
is in fact generated by $k$ small disjoint $2$-spheres $S_1, S_2, \ldots , S_k$ around the $k$ points of the configuration  $\{ p_1, p_2, \ldots , p_k\}$, together with 
a single copy of $\Sigma$ that is homologous to a boundary   component $\Sigma_+$ of $\partial \overline{X}$ in $\overline{X}-\{ p_1, \ldots , p_k\}$. We can therefore 
check our quantization condition by  just evaluating the  integral of $F=\star dV$ on these $k+1$ generators.

    In order to evaluate the corresponding integrals, it  will often be helpful  to pass  to the universal  cover $\mathcal{H}^3$ of $X$, where \eqref{serious} then tells us that 
       $$\star d G_{p_j} = - \frac{1}{2} \sum_{\boldsymbol{\phi} \in \varGamma}\boldsymbol{\phi}^* \alpha ; $$
    here $\alpha$ denotes the pull-back of the standard area form on the unit $2$-sphere $S^2$ in $T_{q_j}\mathcal{H}^3$ 
    via the radial geodesic projection 
    $(\mathcal{H}^3-\{ q_j\})\to S^2$, and  $\boldsymbol{\phi}^* \alpha$ is the pull-back of this singular form via
    the action of $\boldsymbol{\phi}\in \varGamma$ on $\mathcal{H}^3$. By representing the sphere $S_j$ by a small 
    $2$-sphere around $q_j$ that is contained in a fundamental domain for the action, we see that $\star dG_{p_j}$
    restricts to $S_j$ as $-\frac{1}{2} \alpha$ plus an exact form, and that $\star dG_{p_i}$ is exact on $S_j$ for $i\neq j$. 
   We thus  have
   \begin{equation}
\label{punctures}
\frac{1}{2\pi} \int_{S_j} F=\frac{1}{2\pi}\int_{S_j} \star dV = \frac{1}{2\pi}\int_{S^2}\left( -\frac{1}{2}\alpha \right)= - 1\in \ZZ
\end{equation}
    for every $j=1, \ldots , k$, and our quantization condition is therefore automatically  satisfied for the homology generators $[S_1]$, \ldots , $[S_k]$. 
    
    However, the integral of  $\star dG_{p_j}$ on a surface $\Sigma\subset X-\{ p_1, \ldots , p_k\}$ homologous to a boundary component
    $\Sigma_+$ of $\partial \overline{X}$ is a bit more complicated. The answer is best understood in terms of a special harmonic 
    function on $X$ that will  come to play a starring  role in this article:

    \begin{defn} 
    \label{tunfun}
    Let $\varGamma \cong \pi_1(\Sigma)$ be a Bers-type quasi-Fuchsian group,  let $X= \mathcal{H}^3/\varGamma$ be the associated
    hyperbolic $3$-manifold, and let $\overline{X}= [D^3-\Lambda(\varGamma)]/\varGamma$ be the associated $3$-manifold-with-boundary, 
    where $\partial X = [\CP_1- \Lambda (\varGamma)]/\varGamma$. 
    Let $\Sigma_+$ be the component of $\partial X$ on which the boundary orientation agrees with the 
    given orientation of $\Sigma$, 
    and let $\Sigma_-$ be the other component. Then 
    the {\em tunnel-vision function} of $X$ is  defined to be the unique  continuous function ${\zap f}: \overline{X}\to [0,1]$
    which is harmonic on $(X,h)$, equal to $1$ on $\Sigma_+$, and equal to $0$ on $\Sigma_-$.   \end{defn}

 The inspiration for this terminology   also motivates  the proofs of several of our results. Think of $X$ as a tunnel leading from
   $\Sigma_-$ to $\Sigma_+$, and imagine that the tunnel mouth $\Sigma_+$ leads into bright daylight, while $\Sigma_-$ leads into
   darkest night. How big does the bright tunnel opening appear from a point inside the tunnel?  For an observer at $p\in X$, this amounts to asking
   what fraction of the geodesic rays emanating from $p$ end up at $\Sigma_+$, where the measure used to determine this fraction 
   is the usual one on the unit $2$-sphere in $T_pX$. We can understand the answer by passing to the universal cover $\mathcal{H}^3$
   of $X$, and letting $q\in \mathcal{H}^3$ be a preimage of $p$. The sphere at infinity can then be decomposed as
   a disjoint union $\Lambda \sqcup \Omega_+ \sqcup \Omega_-$, where $\Lambda = \Lambda (\varGamma)$ has Lebesgue area
   zero, and where $\Omega_+$ and $\Omega_-$  are  the universal covers of $\Sigma_+$ and $\Sigma_-$, respectively. 
   The question is now equivalent to  asking for the fraction of geodesic rays emanating from $q$ that end up at $\Omega_+$.
   But this fraction is obviously just the  average value  of  the characteristic
   function of ${\Omega_+}$, computed with respect to the area measure  on the sphere at infinity induced by identifying it with the unit sphere in $T_q\mathcal{H}^3$
   via radial projection along geodesics. However, the Poisson integral formula \cite[Chapter 5]{davies} tells us that, as $q$ varies,  
   this spherical average defines a harmonic function $f:\mathcal{H}^3 \to \RR$ that tends to $1$ on $\Omega_+$ and 
   $0$ on $\Omega_-$. Since  $f$ is manifestly $\varGamma$-invariant, it must moreover be the  pull-back of a harmonic function on $X$, and since
   this harmonic function tends to $1$ at $\Sigma_+$ and to $0$ at $\Sigma_-$, it must  therefore coincide with the tunnel-vision function ${\zap f}$. This proves  that
  the apparent area  of the image of $\Sigma_+$, as seen from $p$, divided by the total area $4\pi^2$ of the unit $2$-sphere,  is exactly  ${\zap f}(p)={f}(q)$. 
  In other words, the  value at $p$ of our  tunnel-vision function ${\zap f}$   is equal to what some 
  analysts would call the {\sf harmonic measure $\omega (p, \Sigma_+, X)$
  of $\Sigma_+$ in the space $X$ in with respect to the reference point $p$}.

   This geometric description of the tunnel-vision function ${\zap f}$ has a flip-side that explains why we have chosen to introduce it at this particular juncture:

  \begin{lem} Let $p\in X$, and let $\Sigma\subset X$ be a surface which is homologous to $\Sigma_+$ in $\overline{X}-\{ p\}$. Then
  $$\int_\Sigma \star dG_p = -2\pi {\zap f}(p).$$
  \end{lem}
  \begin{proof}
  First notice that the $2$-form $\star dG_p$ is smooth up to the boundary of $\overline{X}$. Indeed, if we 
  use the upper-half-space
  model  to 
  represent the hyperbolic  metric as $h= (dx^2+dy^2+dz^2)/z^2$
  near 
  some boundary point of $\overline{X}$,  we then have 
   $\star dG_p= {z}^{-1} \starry d G_p$, where $\starry$ is the Hodge star with respect to the Euclidean 
  metric $dx^2+ dy^2 + dz^2$. The fact that $dG_p$ is smooth up to the boundary and vanishes there thus guarantees that $\star dG_p$
  extends smoothly to all of $\overline{X}-\{p\}$. 
  
  Since $\star dG_p$ is consequently a smooth  closed $2$-form on $\overline{X}-\{ p\}$, and because 
  $\Sigma$ is  homologous to $\Sigma_+$ by hypothesis, 
  Stokes' theorem now immediately tells us that  $\int_\Sigma  \star dG_p=\int_{\Sigma_+}  \star dG_p$. To compute the latter integral, we now remember that 
  the universal cover of $\Sigma_+$ is exactly $\Omega_+$. Letting $\dom\subset \Omega_+$
  be a fundamental domain for the action of $\varGamma$ on $\Omega_+$, our expression \eqref{serious} for the pull-back of
  $G_p$ to $\mathcal{H}^3$ therefore tells us that
    \begin{eqnarray*}
 \int_{\Sigma_+} \star dG_p &=& \int_\dom \star d \left( \sum_{\boldsymbol{\phi} \in \varGamma} \frac{1}{e^{2\dist (\boldsymbol{\phi} q_j,q)}-1}\right)
\\&=&
   \sum_{\boldsymbol{\phi} \in \varGamma}  \int_\dom \star d \left(\frac{1}{e^{2\dist (\boldsymbol{\phi} q_j,q)}-1}\right)
 \\&=&
  - \frac{1}{2} \sum_{\boldsymbol{\phi} \in \varGamma}\int_\dom \boldsymbol{\phi}^* \alpha
 \\&=&  - \frac{1}{2} \sum_{\boldsymbol{\phi} \in \varGamma}\int_{\boldsymbol{\phi}(\dom)} \alpha  \\&=&  - \frac{1}{2} \int_{\Omega_+} \alpha . 
\end{eqnarray*}
  where $q$ is a preimage of $p$, and $\alpha$ is the  pull-back of the area form on the unit sphere $S^2\subset T_q$ to $\overline{\mathcal{H}^3}-\{ q\}$
  via geodesic radial projection. However, we have just observed  that the Poisson integral formula tells us
   that $\frac{1}{4\pi}\int_{\Omega_+}\alpha$ is exactly the tunnel-vision function ${\zap f}$ evaluated at
  $p$. It thus follows that 
  $$ \int_\Sigma \star dG_p  = \int_{\Sigma_+} \star dG_p =  - \frac{1}{2} \int_{\Omega_+} \alpha  = - \frac{1}{2}[ 4\pi  {\zap f}(p) ] = - 2\pi {\zap f}(p), $$
 exactly  as claimed. 
    \end{proof}
    
   Adding up $k$ such contributions   now yields a useful corollary:
   
    \begin{lem} 
    \label{quantize}
    Let $\{ p_1, \ldots , p_k\}$ be any configuration of $k$ points in $X$, and let $V$ be the positive potential defined by \eqref{potential}.
    If $\Sigma \subset X- \{ p_1, \ldots , p_k\}$ is a surface homologous to the boundary component $\Sigma_+$ in $\overline{X} - \{ p_1, \ldots , p_k\}$,
    then 
     $$\frac{1}{2\pi} \int_\Sigma \star dV = - \sum_{j=1}^k {\zap f}(p_j).$$
  \end{lem}

   Since ${\zap f}: X\to (0,1)$, 
  it  follows that our quantization condition can never be satisfied if $k=1$. Fortunately, however, this problem  does not reoccur for
   larger values of $k$:
  
  \begin{prop} 
  \label{unobstructed} 
 Let $(X,h)$ be any quasi-Fuchsian hyperbolic $3$-manifold of Bers type. Then for every integer $k\geq 2$, 
    there are quantizable configurations   $\{ p_1 , \ldots , p_k\}$ of $k$ distinct points in  $X$.   
    \end{prop} 
  
  \begin{proof} According to  Definition \ref{intent}, the claim just means that  there are  configurations $\{ p_1, \ldots , p_k\}$ of  distinct points in $X= \mathcal{H}^3/\varGamma$
  for which $\frac{1}{2\pi} \star dV $ represents an element of $H^2(X- \{ p_1, \ldots , p_k\}, \ZZ )\subset H^2 (X- \{ p_1, \ldots , p_k\}, \RR )$ in deRham cohomology. 
  Since $H_2 (X-\{p_1, \ldots  , p_k\})$ is generated by $\Sigma$ and  small  $2$-spheres $S_1, \ldots , S_k$ about the 
  points   $p_1, \ldots , p_k$ of the configuration,
  we only need to arrange for the integrals of $\star dV$ on these generating surfaces  to all be integers. 
  However,  since \eqref{punctures} shows that the integrals on $S_1, \ldots , S_k$ all equal $-1$, we only need to worry about the integral on $\Sigma$,
  which equals $- \sum_{j=1}^k {\zap f}(p_j)$ by Lemma \ref{quantize}. But because  
  ${\zap f}: \overline{X}\to [0,1]$ is continuous and achieves the values $0$ and $1$ exactly on its two boundary components, 
  and because  $\overline{X}$ is connected,    every element of the   interval $(0,1)$ must occur as the value of ${\zap f}$ at some point of $X$. Moreover, since the restriction of
   ${\zap f}$ to $(X,h)$ is harmonic, every such value is attained by uncountably many different points in $X$; indeed,  the mean-value theorem 
   guarantees that ${\zap f}(p)$ also occurs as a value of ${\zap f}$ restricted to  the sphere of radius $\varrho$ about $p$, for every $\varrho$ smaller than the injectivity radius of
   $(X,h)$ at $p$. 
  If $\ell$ is any integer 
   from $1$ to $k-1$,  we can therefore   pick distinct points $p_1 , \ldots , p_k\in X$ with ${\zap f}(p_1) = \cdots = {\zap f}(p_k)= \ell/k$, which  then ensures that 
    $\frac{1}{2\pi}\int_\Sigma \star dV= -\ell$. Of course, the same reasoning also shows that this same goal can also be 
    attained by  specifying the ${\zap f}(p_j)$ to be any other $k$ elements in $(0,1)$  that add up to  $\ell$.
  \end{proof}
 
 In fact, the space  of quantizable configurations is  a real-analytic subvariety of the non-singular part of the $k$-fold symmetric product $X^{[k]}$, and is
 locally cut out by the vanishing of a single harmonic function. 
 However, this space is disconnected if $k\geq 3$, since $\ell=\sum_{j=1}^k {\zap f}(p_j)$   must be an integer for
 every quantizable configuration, and every integer  from $1$ to $k-1$ arises in this manner.

\begin{figure}[htb]
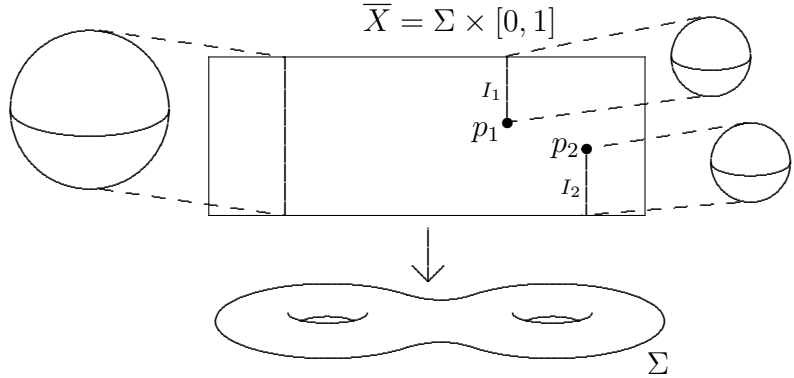

\centerline{
\beginpicture
\setplotarea x from 0 to 320, y from -5 to 140
\putrectangle corners at 75 130 and 240 70
\arrow <6pt> [1,2] from 158  65  to 158 45
\put {\circle*{4}} [B1] at 190 105
\put {\circle*{4}} [B1] at 220 95
\put {$p_1$} [B1] at 180 100
\put {${\scriptstyle I_1}$} [B1] at 182 115
\put {$p_2$} [B1] at 210 94
\put {${\scriptstyle I_2}$} [B1] at 212 78
\put {$\Sigma$} [B1] at 245 10 
\put {$\overline{X}= \Sigma\times [0,1]$} [B1] at 170 140
\circulararc 360 degrees from 60 110 center at 30 110
\ellipticalarc axes ratio 3:1 -180 degrees from 60 110
center at 30 110
\circulararc 360 degrees from 265 115 center at 265 130
\ellipticalarc axes ratio 3:1 -180 degrees from 280 130
center at 265 130
\circulararc 360 degrees from 280 75 center at 280 90
\ellipticalarc axes ratio 3:1 -180 degrees from 295 90
center at 280 90
\ellipticalarc axes ratio 3:1  270 degrees from 150 40
center at 120 30
\ellipticalarc axes ratio 3:1  -270 degrees from 175 40
center at 205 30
\ellipticalarc axes ratio 4:1 -180 degrees from 135 33
center at 120 33
\ellipticalarc axes ratio 4:1 145 degrees from 130 30
center at 120 29
\ellipticalarc axes ratio 4:1 180 degrees from 190 33
center at 205 33
\ellipticalarc axes ratio 4:1 -145 degrees from 195 30
center at 205 29
{\setquadratic 
\plot 150 40    163 38    175 40   /
\plot 150 20    163 22    175 20  /
}
{\setlinear
\plot  184 105 184 130 /
\plot  214 95 214 70 /
\plot  100 130 100 70 /
\setdashes 
\plot  30 140 100 130   /
\plot  30 80  100 70  /
\plot  257 145  182 130   /
\plot  257 115  182 105    /
\plot  272 75  212 70   /
\plot  275 105  212 95   /
}
\endpicture
}
\caption{ \label{ansatz}
One can construct an anti-self-dual $M^4$ from  any 
quasi-Fuchsian  $X^3\approx \Sigma \times (0,1)$ and any  quantizable configuration $\{ p_1, \ldots , p_k\}$
of $k$ points in $X$. The resulting $M$ comes equipped with
an isometric $S^1$-action  such that $M/S^1=\overline{X}\approx \Sigma \times [0,1]$.
This $S^1$-action  action  has fixed points that project to $\{ p_1, \ldots , p_k\}\cup \partial\overline{X}$, but is  free everywhere else. 
If $I_j$ is a segment in $\overline{X}$ that joins the configuration point $p_j$ to  $\partial\overline{X}$ while avoiding  other points
of the configuration,  the inverse image of $I_j$ in $M$ is then a $2$-sphere of self-intersection $-1$. If we choose a disjoint collection $I_1, \ldots, I_k$ of 
$k$  such segments and then collapse the corresponding $2$-spheres in $M$, 
we obtain an $S^2$-bundle over $\Sigma$.  
This last assertion, which implies that $M\approx (\Sigma \times S^2) \# k \overline{\CP}_2$,  is best seen by  first observing 
 that  the pre-image in $M$ of any  segment that joins  the two components of  $\partial\overline{X}$, while avoiding
the configuration,  is a $2$-sphere with trivial normal bundle. 
}
\end{figure}

 With  Proposition \ref{unobstructed} in hand, we now proceed to construct anti-self-dual metrics on $(\Sigma \times S^2)\# k \overline{\CP}_2$  associated with each  
  quasi-Fuchsian hyperbolic manifold $(X,h)$ and each quantizable configuration of $k$ distinct points $p_1, \ldots p_k$ in $X$, in a manner illustrated by
  Figure \ref{ansatz}.  Indeed, given
  a quantizable configuration, let $V$ be given by \eqref{potential}, set $F=\star dV$, and 
  let $P\to X-\{ p_1 , \ldots , p_k\}$ be a principal circle bundle with first Chern class $c_1(P) = [\frac{1}{2\pi}F]\in H^2 (X-\{ p_1 , \ldots , p_k\}, \ZZ )$. 
  Let $\theta$ be a connection $1$-form on $P$ with curvature $d\theta = F$, let ${\zap u} : \overline{X}\to [0,\infty )$ be a non-degenerate
  defining function for $\partial \overline{X}$, and then equip $P$ with the Riemannian metric 
  $$g = {\zap u}^2 \left(Vh+ V^{-1} \theta^2\right).$$
  Because $h$ is hyperbolic, $V$ is harmonic, and  $d\theta = \star dV$, this ``hyperbolic ansatz''
  metric is automatically anti-self-dual \cite{mcp2} with respect to a natural orientation of $P$. 
  Moreover, its metric-space completion is actually a smooth anti-self-dual $4$-manifold $(M,g)$, obtained by adding one extra point $\hat{p}_j$
  for each point $p_j$ of the configuration, and a pair of surfaces $\Sigma_\pm$ conformal to the two components of $\overline{X}$; 
 this can be proved \cite{jongsu,mcp2,leblown,lebsdhg} by explicitly constructing the completion, using local models near $\Sigma_\pm$ and the $\hat{p}_j$. 
 The resulting smooth compact anti-self-dual  $4$-manifolds $(M,g)$ are then all diffeomorphic to $(\Sigma\times S^2) \# k \CP_2$, and carry a conformally
 isometric semi-free $S^1$-action that is generated by a conformal Killing field $\xi$ of period $2\pi$. The invariant $\ell\in \{ 1,\ldots , k-1\}$
 of the configuration now becomes an invariant of the $S^1$-action, because the fixed surfaces $\Sigma_+$ and $\Sigma_-$ of the action now have
 self-intersection numbers $-\ell$ and  $-k+\ell$, respectively.

 However, it is also worth noting that $(M,[g])$ is not uniquely determined by $(X, \{ p_1, \ldots , p_k\})$, because 
 the principal-bundle-with-connection $(P,\theta )$ is not determined up to gauge-equivalence by its curvature $F$. Indeed, if
 $\Sigma$ has genus ${\zap g}\geq 2$, there is a $2{\zap g}$-dimensional torus $H^1(\Sigma , \RR) /H^1(\Sigma , \ZZ)$ of
 flat $S^1$-connections on $X\approx \Sigma \times \RR$, and this torus then acts freely on $S^1$-connections
 over $X- \{ p_1, \ldots , p_k\}$ without changing their curvatures. This  additional freedom in the construction 
 supplements our freedom to choose  $(X,h, \{ p_1, \ldots , p_k\})$, and  generalizes the extra  choice of a 
  flat $S^1$-connection we previously encountered  in  the locally-conformally-flat case. 
 
 \bigskip 
  
 When $\varGamma$ is a Fuchsian group, the anti-self-dual conformal class $[g]$ on $M$ can actually be 
 represented \cite{leblown} by a scalar-flat K\"ahler metric, obtained by using the specific 
 defining function ${\zap u} = \sqrt{f(1-f)}$ for $\partial \overline{X}$ as our conformal factor. However, this does not happen for other quasi-Fuchsian groups: 
  
  \begin{prop}
  Let $(M,[g])$ be an anti-self-dual $4$-manifold arising from a quasi-Fuchsian group $\varGamma$ of Bers type and a (possibly empty) 
  quantizable configuration of  points in $X=\mathcal{H}^3/\varGamma$ via the hyperbolic-ansatz construction. Then $[g]$ is represented by a global scalar-flat 
  K\"ahler metric $g\in [g]$ if and only if $\varGamma$ is a Fuchsian group.
  \end{prop} 
  \begin{proof}  When $\varGamma$  is Fuchsian, it was shown in  \cite{leblown} that the hyperbolic-ansatz construction and an appropriate choice of conformal factor
  yield a scalar-flat K\"ahler metric. 
  In particular, because $[g]$ is represented by a global scalar-flat metric in the Fuchsian case, the constructed conformal class has
   Yamabe constant $Y_{[g]}=0$. By contrast, however, one can show that 
   $Y_{[g]}< 0$ if $\varGamma$ is quasi-Fuchsian but not Fuchsian. Indeed, Jongsu Kim \cite{jongsu}, generalizing a result of Schoen and Yau \cite{schyaudim},
  showed that the Yamabe constant of an anti-self-dual  manifold $(M^4,[g])$ arising from the hyperbolic ansatz is negative whenever the limit set $\Lambda (\varGamma)$ of the 
  corresponding Kleinian group $\varGamma$ has Hausdorff dimension $> 1$. The claim  therefore follows from a result of Bowen \cite{bowen} 
   that  $\dim_H \Lambda (\varGamma ) >  1$ for any a quasi-Fuchsian group of Bers type that  is not 
  Fuchsian; cf. 
  \cite{bishbow,sullivan}.
  \end{proof}

\section{Harmonic Forms and Harmonic Functions}
\label{harmandammer} 

Given a quasi-Fuchsian hyperbolic $3$-manifold $X\approx \Sigma \times \RR$ of Bers type, we have now seen how to  construct 
 anti-self-dual conformal classes $[g]$ on $(\Sigma \times S^2)\# k \overline{\CP}_2$, $k\neq 1$,  from quantizable configurations  of $k$ points in $X$. Because
these $4$-manifolds $M$ all have $b_+=1$, each such $(M,[g])$ carries exactly a $1$-dimensional space
of self-dual harmonic $2$-forms; that is, there is a non-trivial self-dual harmonic $2$-form $\omega$ on any
such $(M,g)$, and this form is   unique up to multiplication by a non-zero real constant. Our next goal
 is to translate the question of whether   $\omega\neq 0$ everywhere into a
question about the quasi-Fuchsian hyperbolic manifold $(X,h)$.

We begin with a local study of the problem. Recall that an open dense set $P$ of $M$ was constructed
as a circle bundle $P\to X- \{ p_1, \ldots ,  p_k\}$, equipped with a connection $1$-form $\theta$ whose
curvature $F=d\theta$ is given by $\star dV$, where $V$ is the positive harmonic function on $(X- \{ p_1, \ldots ,  p_k\},h)$
given by \eqref{potential}. We then equipped $P$ with a conformal class that is represented on 
$P$ by 
\begin{equation}
\label{startup}
g_0= Vh + V^{-1}\theta^{\otimes 2}
\end{equation}
although we have until now generally tended to focus  on  conformally rescalings $g={\zap u}^2g_0$
that were chosen so as to extend to the compact manifold $M$. We emphasize that $(P,g_0)$ carries an
isometric $S^1$-action that  is generated by a Killing field $\xi$ that satisfies $\theta (\xi )\equiv 1$.

Now suppose that $\omega$ is a self-dual $2$-form on $P$ which is invariant under the fixed
isometric $S^1$-action on $P$. Here, we fix our orientation conventions so that, if $e^1, e^2, e^3$
is an oriented orthonormal co-frame on $(X, h)$, then $V^{1/2}e^1, V^{1/2} e^2, V^{1/2} e^3, V^{-1/2}\theta$ is an oriented orthonormal co-frame with respect to $g_0$. It follows that 
$$e^1 \wedge \theta + V e^2 \wedge e^3$$
is a self-dual $2$-form on $(P,g_0)$ of point-wise norm $\sqrt{2}$, and, since $SO(3) \subset SO(4)$
acts transitively on the unit sphere in $\Lambda^+$, this implies that {\em any} $\xi$-invariant 
self-dual $2$-form field of
norm $\sqrt{2}$ can locally be expressed in this way by choosing an appropriate oriented orthonormal frame
on $X$. It therefore follows that any $\xi$-invariant self-dual $2$-form on $P$ can be uniquely written 
as 
\begin{equation}
\label{duck}
\omega = \psi  \wedge \theta + V\star \psi
\end{equation}
for a unique $1$-form $\psi$ on $X-\{ p_1, \ldots , p_k\}$, where $\star$ is the Hodge star of the oriented $3$-manifold $(X,h)$. 

 Let us now also suppose that the self-dual $2$-form $\omega$ is also {\em closed}; of course, since $\omega = \ast \omega$,  where 
$\ast$ denotes the $4$-dimensional Hodge star, this then implies that $\omega$ is co-closed,  and hence harmonic. 
Since $\omega$ is invariant under the flow of $\xi$,  Cartan's magic  formula for the Lie derivative of a differential form therefore tells us that 
$$0 =\mathcal{L}_\xi \omega = \xi \lrcorner ~d\omega  + d (\xi \lrcorner~ \omega) = - d\psi$$
so that $\psi$ must be a closed $1$-form on $X-\{ p_1, \ldots , p_k\}$.  However, if we  let $\mu_h$ denote the volume $3$-form of $(X,h)$,  
we also have 
\begin{eqnarray*}
0&=& d\omega\\ &=& -\psi \wedge d \theta + d(V\star \psi) 
\\ &=& -\psi \wedge \star dV + dV\wedge \star \psi + V (d\star \psi )
\\&=&  -\langle \psi , dV \rangle \mu_h + \langle dV , \psi \rangle \mu_h - V (d\star \psi )
\\ & =& -V (d\star \psi )
\end{eqnarray*}
and we therefore conclude that the $1$-form $\psi$ is {\em strongly harmonic}, in the sense that 
$$ d\psi =0, \qquad d \star \psi = 0.$$
Conversely, if $\psi$ is any strongly harmonic $1$-form  on $X-\{ p_1, \ldots , p_k\}$, the 
 $2$-form $\omega$ defined on $(P,g)$  by \eqref{duck} is closed and self-dual, and hence harmonic. 
 We note in passing that this argument does not  depend on the fact that 
 $h$ has constant curvature or that $(P,g_0)$ is anti-self-dual; one just needs $g_0$ to be expressed as \eqref{startup} for some metric $h$, some 
  positive harmonic function 
 $V$,  and some connection $1$-form $\theta$ on $P$ with curvature $F=\star dV$. 
 
 Notice that the relationship between $\omega$ and $\psi$ codified by \eqref{duck}  entails a simple
 relationship between the point-wise norms of these forms. Indeed, notice that 
$$
\omega \wedge \omega = 2  \psi \wedge  \theta \wedge V \star \psi = 2 V (\psi \wedge \star \psi ) \wedge \theta 
= 2 V |\psi |^2_h ~\mu_h \wedge \theta~.
$$
On the other hand, since $g_0$ has volume form 
$$
\mu_{g_0} = V^{1/2}e^1\wedge  V^{1/2} e^2 \wedge  V^{1/2} e^3 \wedge  V^{-1/2}\theta = V \mu_h \wedge \theta
$$
the self-duality of $\omega$ thus implies that 
$$
|\omega |_{g_0}^2 \mu_{g_0} = \omega \wedge\omega = 2 | \psi|^2_h \mu_{g_0}
$$
and hence that 
\begin{equation}
\label{norman}
|\omega |_g = \sqrt{2} |\psi|_h .
\end{equation}
This will allow us to invoke the following removable singularities result:

\begin{lem}
\label{removable} 
Let $p$ be a point of a  smooth, oriented 
  Riemannian $n$-manifold $(Y, \mathbf{g})$, $n \geq 2$, and let $\varphi$ be a differential $\ell$-form 
on $Y-\{ p\}$ that is {\sf strongly harmonic}, in the sense that $d\varphi=0$ and $d\star \varphi =0$. 
 Also suppose that $\varphi$ is bounded near $p$,  in the sense that there 
is a neighborhood $U$ of $p$ and a positive constant $C$ such that the point-wise norm $|\varphi|$ satisfies
$|\varphi | < C$ on $U-\{ p\}$. Then $\varphi$  extends  uniquely to $Y$ as a smooth strongly harmonic 
$\ell$-form. 
\end{lem}
\begin{proof}
Let $\alpha$ be any smooth, compactly supported $(n-\ell-1)$-form on $Y$. Letting 
$B_\epsilon$ denote the $\epsilon$-ball around $p$ for any small $\epsilon$, and setting $S_\epsilon= \partial B_\epsilon$, we then have 
$$
\int_Y  \varphi \wedge d\alpha = \lim_{\epsilon\to 0} \int_{Y-B_\epsilon} \varphi \wedge d\alpha = \pm  \lim_{\epsilon\to 0} \int_{Y-B_\epsilon} d(\varphi \wedge \alpha )
 =\mp \lim_{\epsilon\to 0} \int_{S_\epsilon} \varphi \wedge \alpha  =  0,
$$
because $\varphi \wedge \alpha$ is bounded,  and the area of $S_\epsilon$ tends to zero as $\epsilon \to 0$. Hence 
the $L^\infty$ form $\varphi$ satisfies  $d\varphi =0$ in the   sense of  currents.
Similarly,  if $\beta$ is any smooth, compactly supported 
$(\ell-1)$-form on $Y$, then 
$$
\int_Y  \varphi  \wedge \star d\beta   = \lim_{\epsilon\to 0} \int_{Y-B_\epsilon}   (d\beta ) \wedge \star \varphi = \lim_{\epsilon\to 0} \int_{Y-B_\epsilon}   d(\beta  \wedge \star \varphi )= -\lim_{\epsilon\to 0} \int_{S_\epsilon} \beta  \wedge \star \varphi  =  0,
$$
so that $d (\star \varphi ) =0$ in the sense of currents, too. Thus $\Delta \varphi =0$ in the distributional sense, and  elliptic regularity then guarantees that $\varphi$ is a smooth
 $\ell$-form on $Y$. Since 
$\varphi$ is both closed and co-closed on the open  dense subset  $Y-\{ p\}$, it therefore follows by continuity
that its extension is also closed and co-closed on all of  $Y$. \end{proof}

We now restrict our attention to the problem at hand. Let $(M,[g])$ be a smooth compact $4$-manifold produced from a quasi-Fuschsian hyperbolic 
$3$-manifold $(X,h)$ and a (possibly empty) quantizable configuration of points $\{ p_1, \ldots , p_k\}$ by the hyperbolic-ansatz construction. Thus, 
$(M,[g])$ comes equipped with a conformally  isometric $S^1$-action such that $\overline{X}= M/S^1$, and such that
$P\to X- \{ p_1, \ldots , p_k\}$ is the union of the free  $S^1$-orbits. We can thus represent $[g]$ by a smooth metric of the form
$g={\zap u}^2 g_0$, where $g_0$ is given by \eqref{startup}, and where ${\zap u} : \overline{X}\to [0,\infty )$ is a smooth non-degenerate defining function for
$\partial \overline{X} = \Sigma_-\sqcup \Sigma_+$. Let $\omega$ be a non-trivial self-dual $2$-form on $(M,[g])$. In particular, the restriction of
$\omega$ to the dense subset $P$ is also non-trivial, and our previous calculations then show that $\psi = -\xi \lrcorner \omega$ therefore defines a strongly 
harmonic $1$-form on $(X -\{ p_1, \ldots , p_k\}, h)$. However,  $|\omega|_g$ is  bounded on $M$ by compactness, so $|\omega|_{g_0}= {\zap u}^2 |\omega|_g$
is  uniformly bounded on $P$. Equation \eqref{norman} therefore tells us  that $|\psi|_h= |\omega|_{g_0}/\sqrt{2}$  is uniformly bounded on 
$X-\{ p_1, \ldots  p_k\}$, and  $\psi$ consequently extends to all of $X$ as a strongly harmonic $1$-form by Lemma \ref{removable}. 

However, even more is true. Notice that we can define a smooth Riemannian metric on $\overline{X}$ by $\bar{h}:={\zap u}^2 h$, and we then have
$|\psi|_{\bar{h}}= {\zap u}^{-1} |\psi|_h = {\zap u}^{-1}|\omega|_{g_0}/\sqrt{2}= {\zap u}|\omega|_{g}/\sqrt{2}$. This shows that $\psi$ has a continuous
extension to the boundary of $\overline{X}$ by zero. Now notice that any loop in $X\approx \Sigma \times (0,1)$ is freely homotopic to a loop that is arbitrarily close to $\partial
\overline{X}$, and on which the integral of $\psi$ is therefore as small as we like. But  $\psi$ is closed, and its integral on a loop is therefore
invariant under free homotopy. This shows that the integral of $\psi$ on any loop must vanish, and that $[\psi]\in H^1(X,\RR)$ therefore vanishes. 
Thus $\psi$ is exact, and we therefore have $\psi = d{f}$ for some smooth function on $X$. Moreover, since $d\star \psi =0$, we have
$\Delta {f}=-\star d\star d {f} = 0$, so ${f}$ is therefore a harmonic function on $(X,h)$. Since we can explicitly construct
${f}$ from $\psi$ by integration along paths from a base-point, and since $\psi \to 0$ at $\partial X$, this harmonic function on $X$
tends to a constant on each boundary component, and, since $df=\psi\not\equiv 0$,  the values on the two boundary components must moreover be different by 
the maximum principle. By adding a constant if necessary, we can now arrange for ${f}$ to tend to zero  at $\Sigma_-$, and, at the price
of perhaps replacing $\omega$ with a constant multiple, we can then also arrange for ${f}$ to tend to $1$ along $\Sigma_+$. This means
that we have arranged for ${f}$ to exactly be the tunnel-vision function ${\zap f}$ of Definition \ref{tunfun}. This proves the following result:

\begin{prop} Let $(M,[g])$ be an anti-self-dual $4$-manifold arising via the hyperbolic ansatz from a quasi-Fuchsian hyperbolic $3$-manifold $(X,h)$ of Bers type
and a (possibly empty) quantizable configuration $\{ p_1, \ldots , p_k\}$ of points in $X$. Then any  self-dual harmonic form on $(M,[g])$ 
restricts to the open dense subset $P\subset M$ as a constant multiple of 
$$
\omega : = d{\zap f}\wedge \theta  + V \star d{\zap f}, 
$$
where ${\zap f}$ and $\star$ are respectively  the tunnel-vision function and Hodge star of the quasi-Fuchsian hyperbolic $3$-manifold $(X,h)$,   while 
$V$ is the potential assigned to  the configuration $\{ p_1, \ldots , p_k\}$ by \eqref{potential}, and $\theta$ is the connection $1$-form with $d\theta= \star dV$
used to construct $[g]$ via  \eqref{ansatz0}. 
\end{prop}

To fully exploit this observation, however, we will still need one other key fact about the tunnel-vision function:

\begin{lem} 
\label{expansion}
Let  $(X,h)$ be a quasi-Fuchsian hyperbolic $3$-manifold of Bers type, and let ${\zap f}: \overline{X}\to [0,1]$ be its
tunnel-vision function. Then at every point of $\partial \overline{X}$,  the first normal derivative of ${\zap f}$ is zero, but the 
second normal derivative of ${\zap f}$ is non-zero.
\end{lem}
\begin{proof} Recall that $\partial \overline{X}= \Sigma_+ \sqcup \Sigma_-$. 
It will suffice to show that
\begin{itemize}
\item near any point of $\Sigma_-$, there are two local non-degenerate  local defining functions $u$ and $\tilde{u}$ for 
$\partial \overline{X}$ such that  ${u}^2 \leq {\zap f} \leq \tilde{u}^2$   near the given point; and that 
\item near any point of $\Sigma_+$, we can 
similarly find two local non-degenerate  local defining functions $u$ and $\tilde{u}$ such that ${u}^2 \leq 1-{\zap f} \leq \tilde{u}^2$  near the given point. 
\end{itemize}

Let us first see what happens in the Fuchsian case. Here, the decomposition $\CP_1 = \Omega_+\sqcup \Omega_- \sqcup \Lambda$
is just the decomposition of the sphere at infinity  into a geometric circle and two geometric open disks. In the upper half-space
model, we can thus take $\Omega_+$ and $\Omega_-$  to be the halves  of the $xy$-plane respectively given  by $y > 0$ and $y< 0$. 
In this prototypical situation, the tunnel-vision function just pulls back to become 
$${\zap f}_0= \frac{1}{2}\left( \frac{y}{\sqrt{y^2+z^2}}+ 1\right)$$
which is  harmonic  on $z> 0$ with respect to $h= (dx^2+dy^2+dz^2) /z^2$,  equals $1$ when $y> 0$ and $z=0$,  and equals $0$ when 
$y< 0$ and $z=0$. Now notice that,  when  $z$ is small,  ${\zap f}_0 = (z/2y)^2+ O( (z/y)^3)$ when $y< 0$, while $1-{\zap f}_0 = (z/2y)^2+ O( (z/y)^3)$
when $y > 0$. It  thus  follows that $\sqrt{{\zap f}_0}$ and $\sqrt{1-{\zap f}_0}$ are themselves smooth non-degenerate defining functions for
these boundary half-planes, and we are thus free to take $u=\tilde{u}$ to be these defining functions to emphasize that the claim is certainly
true in this prototypical case. 

Now, in  the general quasi-Fuchsian case, we again have $\CP_1 = \Omega_+\sqcup \Omega_- \sqcup \Lambda$, but $\Lambda$ will 
 just be a quasi-circle, and the open sets $\Omega_\pm$ could be dauntingly complicated. However, if ${\zap p}$ is any point of $\Sigma_+$, we can 
still represent it in the universal cover by some ${\zap q}\in \Omega_+$, and, since $\Omega_+$ is open in $\CP_1$, we may choose some closed geometric disk 
$D_+$ such that ${\zap y}\in D_+ \subset \Omega_+$; moreover, since $\Omega_-$ is also open, 
we can also choose a second closed disk $D_-\subset\CP_1$ such that 
$\Omega_+ \subset D_-$ by taking $\CP_1- D_-$ to be a small open disk around some $\tilde{\zap q}\in\Omega_-$.  
Now let ${\zap f}_\pm$ be the harmonic functions on $\mathcal{H}^3$ whose values at $p$ are  the average values of the characteristic
functions of $D_\pm$ with respect to the visual measure at $p$. We then immediately have 
\begin{equation}
\label{shopper}
{\zap f}_+  \leq {\zap f} \leq {\zap f}_-
\end{equation}
everywhere, because $D_+\subset \Omega_+ \subset D_-$.  On the other hand, our discussion of the Fuchsian case shows that
$u = \sqrt{1-{\zap f}_-}$ and $\tilde{u}= \sqrt{1-{\zap f}_+}$ are non-degenerate defining functions for $S^2= \partial \overline{\mathcal{H}^3}$ near ${\zap q}$, and 
our last inequality then becomes 
$${u}^2 \leq 1 - {\zap f} \leq \tilde{u}^2$$
as desired. 
On the other hand, if we instead take $\tilde{\zap q}\in \Omega_-$ to represent a given point $\tilde{\zap p}\in \Sigma_-$, our Fuchsian discussion
shows that $u=\sqrt{{\zap f}_+}$ and $\tilde{u}= \sqrt{{\zap f}_-}$ are non-degenerate defining functions for $S^2= \partial \overline{\mathcal{H}^3}$ near $\tilde{\zap q}$,
and we then have 
$${u}^2 \leq  {\zap f} \leq \tilde{u}^2, $$
exactly as required. This shows that ${\zap f}$ has vanishing first normal derivative but non-zero second normal derivative at every point of $\partial \overline{X}$, as claimed. 
\end{proof}

This now allows us to prove  the main result of this section. 

\begin{thm}
\label{pika}
Let $(M,[g])$ be an anti-self-dual $4$-manifold arising via the hyperbolic ansatz from a quasi-Fuchsian $3$-manifold $(X,h)$ of Bers type
and a (possibly empty) quantizable configuration of points in $X$. Then the anti-self-dual conformal class $[g]$ contains 
an almost-K\"ahler metric $g\in [g]$ if and only if the tunnel-vision function ${\zap f} : X\to (0,1)$ has no critical points. 
\end{thm}
\begin{proof} The conformal class $[g]$ contains an almost-K\"ahler representative iff the non-trivial self-dual harmonic form $\omega$ satisfies
$\omega \neq 0$ everywhere. We have just shown that, possibly after multiplying $\omega$ by a non-zero constant, we may assume
that it is associated with the $1$-form $\psi= d{\zap f}$ on $\overline{X}$. Since this means that $\sqrt{2} |d{\zap f}|_{\bar{h}}= {\zap u}|\omega|_{g}$
with respect to any $S^1$-invariant metric in the conformal class $[g]$, a necessary condition for $\omega$ to be everywhere non-zero
is that we must have $d{\zap f}\neq 0$ away from $\partial \overline{X}$. Conversely, if ${\zap f}$ has no
critical points in $X$, the same calculation implies that $\omega$ must be non-zero away from the surfaces $\Sigma_+$ and $\Sigma_-$. 
On the other hand, regardless of the detailed behavior of ${\zap f}$, Lemma \ref{expansion} 
 shows that ${\zap u}^{-1}|d{\zap f}|_{\bar{h}}$ always has non-zero limit at every point of $\partial \overline{X}$,
so we {\em always} have $\omega\neq 0$ at every point on the surfaces $\Sigma_\pm$. 
 This shows  that $\omega\neq 0$ on all of $M$ unless  the tunnel-vision function ${\zap f}$ has a critical point somewhere in
  $X\subset \overline{X}$. 
\end{proof}

\section{Tunnel-Vision Critical Points}

Theorems \ref{main1} and \ref{main2} will now follow
from Theorem \ref{pika}
 if we can produce an appropriate  sequence of   quasi-Fuchsian hyperbolic $3$-manifolds $(X,h)$ 
whose   tunnel-vision functions ${\zap f}: X\to (0,1)$ have critical points. 
The first step is to show that any Jordan curve can be approximated
by the limit set of a suitable quasi-Fuchsian group $\varGamma$. Our 
proof is based on the measurable Riemann mapping theorem in this section,  even though 
 a more elementary and constructive proof  can be given using    concrete reflection
groups.
This lemma  is sometimes attributed to Sullivan and
Thurston  \cite{sulthur}, who used the idea  to construct  4-manifolds with unusual affine structures.

In what follows, we will work in the upper-half-space model of $\mathcal{H}^3$, so that $\CC= \RR^2$ will represent
the complement of a point in the sphere at infinity, even though, as a matter of convention, we will  find it convenient
to endow it with  its usual Euclidean 
metric. The latter will in particular allow us to speak of the   Hausdorff distance between
 two compact subsets,  meaning by definition   the infimum of 
all $\epsilon>0$ so that each set is contained in the 
$\epsilon$-neighborhood of the other.

\begin{lem}\label{closer}
For any piecewise smooth Jordan curve $\gamma\subset \CC$   and any $\varepsilon >0$,  there is a positive integer $N$ such that, for every compact oriented surface
$\Sigma$ of genus ${\zap g} \geq N$, 
there is quasi-Fuchsian group $\varGamma\cong \pi_1(\Sigma )$ of Bers type whose limit 
set $\Lambda ( \varGamma)\subset \CC\subset \CP_1$ is a quasi-circle whose   Hausdorff distance  from $\gamma$ is less than $\varepsilon$. 
Moreover, if $\gamma$ is invariant under  $\zeta
\mapsto -\zeta$, and if ${\zap g}$ is even, we can arrange for $\Lambda (\varGamma )$  to also 
be invariant under reflection  through the origin. 
\end{lem}

\begin{proof} 
Since any piecewise smooth Jordan curve can be uniformly approximated by {smooth}  ones, 
we may  assume for simplicity that  the given Jordan curve $\gamma$ is actually  smooth. With this proviso, the 
Riemann mapping theorem allows us to construct a diffeomorphism
 $\Psi : \CC \to \CC$  that is holomorphic  outside 
the unit disk and maps the unit circle $\TT =\{ |\zeta|=1\}$ 
 to $\gamma$. The diffeomorphism $\Phi : \CC \to \CC$ defined by 
$\Phi (\zeta) = \Psi (\zeta/[1-\epsilon])$ is then  holomorphic outside  the disk $D=\{|\zeta|< 1-\epsilon\}$ 
and maps the unit circle to an approximation $\gamma_\epsilon =
\Phi (\TT)$ of $\gamma$ whose Hausdorff distance from $\gamma$ may be taken   to be  smaller than $\varepsilon/2$ by choosing  
 $\epsilon$ to be sufficiently  small. 
 
 Given any $\delta \in (0,  \epsilon )$, we now construct some $\TT$-preserving Fuchsian groups  with fundamental domains 
 containing the disk  $\{ |\zeta| < 1-\delta\}$. To this end, endow the open unit disk in $\CC$ with  the hyperbolic metric
 $4|d\zeta|^2/(1-|\zeta|^2)^2$, and, for an arbitrary positive integer ${\zap g}\geq 2$,  let  ${\zap P}$ be the regular hyperbolic $4{\zap g}$-gon whose 
  vertices are all equidistant from $0$, and whose interior angles 
  at these vertices  are all equal  to $\pi/2{\zap g}$. By drawing
  geodesic segments from $0$ to the $4{\zap g}$ vertices and  the $4{\zap g}$  midpoints of the sides of ${\zap P}$,
  we can then dissect ${\zap P}$ into $8{\zap g}$  hyperbolic isosceles right triangles, with interior angles $(\pi/2, \pi/4{\zap g}, \pi/4{\zap g})$. 
  Now label the oriented edges of ${\zap P}$ as $a_1$, $b_1$, $a_1^{-1}$, $b_1^{-1}$, \ldots , $a_{\zap g}$, $b_{\zap g}$, ${a_{\zap g}}^{-1}$, ${b_{\zap g}}^{-1}$,
  starting at some reference point and proceeding counter-clockwise, as indicated in Figure \ref{polygon}.  We can then construct  a genus-${\zap g}$ hyperbolic surface 
  $\Sigma$ by identifying the edges of ${\zap P}$ in pairs according to this labeling scheme. The universal cover of $\Sigma$ then becomes
  the open unit disk, and the fundamental group  $\pi_1(\Sigma )$ is then represented as
   a $\TT$-preserving Fuchsian group $\Gamma_{\zap g}$ with fundamental domain ${\zap P}$, 
  where the relevant deck transformations form a finite-index subgroup of the $(2, 4{\zap g}, 4{\zap g})$ triangle group that is generated by reflections 
  through the sides of the isosceles right triangles into which we dissected ${\zap P}$. Now let $\rad$  denote the hyperbolic distance from $0$
  to the midpoint of a side of ${\zap P}$, and let $\Rad$ denote the hyperbolic distance from $0$ to a vertex of ${\zap P}$. Then  ${\zap P}$
  contains the disk of hyperbolic radius $\rad$ about $0$, and is contained in the disk of hyperbolic radius $\Rad$ with the same center. 
 Moreover, since our isosceles right triangles have sides $\rad$, $\rad$, and $\Rad$, we have $\Rad < 2\rad$ by the triangle inequality.  
  Since the hyperbolic area of ${\zap P}$ is $4\pi ({\zap g} -1)$, and  since ${\zap P}$ is contained in a disk of hyperbolic radius $\Rad$, 
  with  hyperbolic area $2\pi (\cosh \Rad -1)< \pi e^\Rad$, it thus follows that   $4 ({\zap g} -1) < e^{\Rad} < e^{2\rad}$. On the other hand, if $\mathbf{r}$ is
  the Euclidean radius of the  disk of hyperbolic radius $\rad$, we have $\rad = \log (1+\mathbf{r})/(1-\mathbf{r})$. Thus 
  $$\frac{2}{1-\mathbf{r}} >  \frac{1+\mathbf{r}}{1-\mathbf{r}} >  2\sqrt{{\zap g}-1}$$
    and hence 
    $$\mathbf{r} > 1- \frac{1}{\sqrt{{\zap g}-1}}.$$
   This shows that  if 
    $${\zap g} \geq  N(\delta ) := 1 + \lceil \frac{1}{\delta^2}\rceil, $$
  our  fundamental domain ${\zap P}$ for the Fuchsian group  $\Gamma_{\zap g} \cong \pi_1(\Sigma )$ will contain the disk of Euclidean radius $1-\delta$ about $0$.

\begin{figure}[htb]
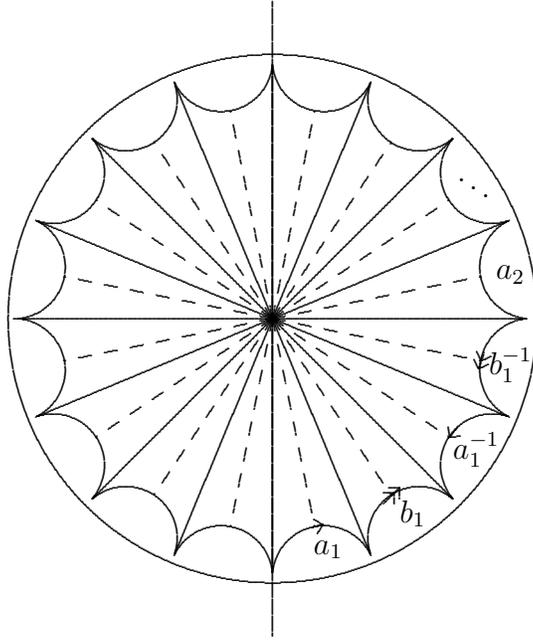

\centerline{
\beginpicture
\setplotarea x from -100 to 100, y from -110 to 100
\circulararc 360 degrees from  0 -100  center at 0 0
\circulararc -149  degrees from  96 0 center at 98 19.5
\circulararc 149  degrees from  96 0 center at 98 -19.5
\circulararc 149  degrees from   0 96 center at  19.5 98
\circulararc -149  degrees from   0 96 center at  -19.5 98
\circulararc 149  degrees from  -96 0 center at -98 19.5
\circulararc -149  degrees from  -96 0 center at -98 -19.5
\circulararc -149  degrees from   0 -96 center at  19.5 -98
\circulararc 149  degrees from   0 -96 center at  -19.5 -98
\circulararc -149  degrees from  89 37 center at 83 55.5
\circulararc 149  degrees from  37 89  center at  55.5 83
\circulararc -149  degrees from  -89 -37 center at -83 -55.5
\circulararc 149  degrees from  -37 -89  center at  -55.5 -83
\circulararc 149  degrees from  89 -37 center at 83 -55.5
\circulararc -149  degrees from  37 -89  center at  55.5 -83
\circulararc 149  degrees from  -89 37 center at -83 55.5
\circulararc -149  degrees from  -37 89  center at  -55.5 83
\put {$a_1$} [B1] at 21 -89
\put {$b_1$} [B1] at 53 -77
\put {$a_1^{-1}$} [B1] at 77 -53
\put {$b_1^{-1}$} [B1] at 90 -21 
\put {$a_2$} [B1] at 90 15
\put {$\ddots$} [B1] at 76 45
\arrow <3pt> [1,2] from 15 -79 to 19 -78
\arrow <3pt> [1,2] from 45 -67 to 48 -64
\arrow <3pt> [1,2] from 42 -70 to 46 -66
\arrow <3pt> [1,2] from   70 -41 to 67 -45
\arrow <3pt> [1,2] from   79  -16 to 78 -19
\arrow <3pt> [1,2] from   80  -13 to 79 -16
{\setlinear 
\plot  0 -120 0 120  /
\plot  0 0 96 0  /
\plot  0 0  -96 0 /
\plot  0 96  0 0 /
\plot  0 -96  0 0 /
\plot  0 0 68 68  /
\plot  0 0 -68 -68  /
\plot  0 0 68 -68  /
\plot  0 0 -68 68  /
\plot  0 0 37  89  /
\plot  0 0 -37  -89  /
\plot  0 0 37  -89  /
\plot  0 0 -37  89  /
\plot  0 0  89 37 /
\plot  0 0  -89 -37 /
\plot  0 0  89 -37 /
\plot  0 0  -89 37 /
\setdashes
\plot  0 0 79 16  / 
\plot  0 0 -79 -16  /
\plot  0 0 79 -16  /
\plot  0 0 -79 16  /
\plot  0 0 16  79  /
\plot  0 0 -16  -79  /
\plot  0 0 -16  79  /
\plot  0 0 16  -79  /
\plot  0 0 67 45  /
\plot  0 0  45 67 /
\plot  0 0 -67 -45  /
\plot  0 0  -45 -67 /
\plot  0 0 67 -45  /
\plot  0 0  45 -67 /
\plot  0 0 -67 45  /
\plot  0 0  -45 67 /
}
\endpicture}
\caption{ \label{polygon}
One can construct  an explicit  hyperbolic metric on a surface $\Sigma$ 
 of genus ${\zap g}\geq 2$ by starting with a  regular hyperbolic $4{\zap g}$-gon ${\zap P}$ whose internal angles equal to $\pi/2{\zap g}$,
and then identifying the sides of ${\zap P}$ in the  manner indicated. If the hyperbolic plane is represented by the
the unit disk,  equipped with the Poincar\'e metric, the Euclidean radius of the in-circle of   ${\zap P}$  approaches $1$ as the 
genus ${\zap g}$ tends to infinity. 
}
\end{figure}

Now assume that ${\zap g} \geq  N(\delta )$, and let $\mu = \Phi_{\bar{\zeta}}/\Phi_\zeta$ be the complex dilatation of $\Phi$. 
Since  $\mu$ is  supported in the disk $D$ of radius $1-\delta$, and since our assumption  guarantees that $D\subset {\zap P}$,
it now  follows that  $\mu$ is supported in the fundamental domain ${\zap P}$ of
$\Gamma_{\zap g}$. We can therefore  \cite[Chapter VI]{ahlfors}
  extend $\mu$ uniquely as a $\Gamma_{\zap g}$-equivariant bounded measurable function $\mu_\delta$ 
on $\CC$ which is supported in the unit disk and has $L^\infty$ norm $<1$; 
namely, 
we first decompose the unit disk as $\cup_{\phi\in \varGamma_{\zap g}}\phi({\zap P})$, then set 
$\mu_\delta := [(\frac{\bar{\phi^\prime}}{\phi^\prime}) \mu]\circ \phi^{-1}$ on each $\phi({\zap P})$, and 
finally  declare that $\mu_{\delta}=0$ outside the unit disk. 
The Measurable Riemann Mapping Theorem \cite{A-B:Beltrami} 
then 
guarantees the existence of a  quasiconformal map $\Phi_\delta$ with dilatation equal to 
 $\mu_\delta$, and the  $\Gamma_{\zap g}$-equivariance of  $\mu_\delta$ moreover guarantees that $\varGamma_\delta = \Phi_\delta \circ \Gamma_{\zap g} \circ \Phi_\delta^{-1}$
is a quasi-Fuchsian group of Bers type.  However,  because the region where $\mu \ne \mu_\delta$  is contained in the annulus
$\{ \zeta\,|\,1-\delta < |\zeta| < 1\}$, it follows that  $\Phi_\delta \to \Phi$ on compact subsets of the unit disk 
as $\delta \searrow 0$. The limit set $\Lambda_\delta = \Phi_\delta (\TT )$ therefore converges
to $\Phi (\TT ) =\gamma_\epsilon$ in the Hausdorff metric as we decrease $\delta$. By choosing $\delta$ sufficiently small,  we thus can 
arrange for $\Lambda_\delta$ to be within Hausdorff distance $\varepsilon/2$ of $\gamma_\epsilon$, and hence within  Hausdorff distance $\varepsilon$
of $\gamma$, as desired.

If $\gamma$ is invariant under reflection through the origin, then 
$\Psi$ and  $\Phi$ can  be chosen to share this symmetry, and  the dilatation $\mu$ is  then consequently  reflection-invariant, too. 
 The line joining  our reference point  to the origin now separates the sides of  ${\zap P}$ into two counter-clockwise lists of $2{\zap g}$ sides.
 If ${\zap g}$ is even, the number of entries on each of these lists is divisible by $4$, and since our listing of the sides as \ldots, $a_j$, $b_j$, $a_j^{-1}$, $b_j^{-1}$, \ldots
 breaks them into quadruples, our rules for identifying the sides  do not mix  the two lists. 
 Since  reflection through the origin is the same as a $180^\circ$ rotation, this involution  compatibly intertwines with 
our  rules for identifying the sides of ${\zap P}$ to obtain $\Sigma$, and so 
induces an  orientation-preserving isometry of the hyperbolic surface $\Sigma$ with two fix points --- namely, the origin and the equivalence class consisting of all the vertices
of ${\zap P}$. 
It follows that 
$\Gamma_{\zap g}$ and  reflection  through the origin  generate a group extension  
$$1\to \Gamma_{\zap g}\rightarrow  \widehat{\Gamma}_{\zap g} \to \ZZ_2\to 0,$$
where   $\widehat{\Gamma}_{\zap g}$ is a larger Fuchsian group, 
but 
is no longer torsion-free.  
Notice that each  of the halves
  into which ${\zap P}$ is divided by   our reference line is now a fundamental domain for the action of $\widehat{\Gamma}_{\zap g}$ on the unit disk. 
The reflection symmetry of $\mu$  thus guarantees that our $\Gamma_{\zap g}$-equivariant extension  $\mu_\delta$ of $\mu$  is actually $\widehat{\Gamma}_{\zap g}$-equivariant,
so that $\widehat{\varGamma}_\delta :=\Phi_\delta\circ \widehat{\Gamma}_{\zap g} \circ \Phi_\delta^{-1}$ 
 defines a  larger quasi-Fuchsian group.   
However, $\Lambda_\delta$ is  also  the limit set of the group $\widehat{\varGamma}_\delta$. On the other hand, 
  $\widehat{\varGamma}_\delta$ is
exactly the group generated by ${\varGamma}_\delta$ and reflection through the origin. Thus, in this situation, all of the constructed 
 limit sets  $\Lambda_\delta$  approximating  $\gamma$ are  
reflection-invariant, and all our claims have therefore been proved. 
 \end{proof}

When  $\varGamma$ is a quasi-Fuchsian group of Bers type, recall that 
 we have defined the tunnel-vision function ${\zap f}$ of $X=\mathcal{H}^3/\varGamma$
 to be the unique harmonic function on $X$ which tends to $0$ at one  component $\Sigma_-$ of $\partial\overline{X}$  and tends to $1$ at the other
component $\Sigma_+$ of $\partial\overline{X}$. 
The limit set $\Lambda= \Lambda (\varGamma)$ then divides $\CP_1$  into two connected components  
$\Omega_-$ and $\Omega_+$, which may be respectively identified with the universal covers of $\Sigma_-$ and $\Sigma_+$
and the pull-back   $\widetilde{\zap f}$ of ${\zap f}$ to   $\mathcal{H}^3$ then tends to 
$0$ at $\Omega_-$ and tends to $1$ at $\Omega_+$.  
 This means that we can reconstruct $\widetilde{\zap f}$ from 
the open set $\Omega_+$ using the {\em Poisson kernel} of $\mathcal{H}^3$. This 
$2$-form on the sphere at infinity, depending on a point in hyperbolic $3$-space, 
is explicitly given  \cite{davies} in the  upper half-space model by $P_{(x,y,z)} d\xi\wedge d\eta$, where 
\begin{equation}
\label{fishtale}
P_{(x,y,z)}(\xi,\eta) = \frac{1}{\pi} \left [ \frac {z}{(x-\xi)^2+(y-\eta)^2+z^2} \right ]^2 .
\end{equation}
Note that  the factor of $1/\pi$ is  inserted here to make the form have total mass $1$.
Up to a constant factor,  $P_{(x,y,z)}d\xi \wedge d\eta$ is 
just  the ``visual area form'' obtained by identifying the
unit sphere in $T_{(x,y,z)}\mathcal{H}^3$  with the sphere at infinity by 
following geodesic rays starting at $(x,y,z)$. 
Given a bounded measurable function $F(\xi, \eta)$ on the sphere at infinity, we can uniquely extend it  to $\mathcal{H}^3$ 
as a bounded harmonic function $f$ via the hyperbolic-space version of the {\em Poisson integral formula} 
\begin{equation}
\label{fishdish}
f(x,y,z) = \int_{\RR^2} F(\xi,\eta) P_{(x,y,z)}(\xi,\eta)  \thinspace d\xi \wedge d\eta ; 
\end{equation}
for more details, see \cite[Chapter 5]{davies}.
For us, the importance of this formula stems from the fact  that when $F=\chi_{\Omega_+}$ is the indicator function of $\Omega_+$,  
the resulting  $f$ is exactly the pull-back $\widetilde{\zap f}$ of the tunnel-vision function 
of  $X= \mathcal{H}^3/\varGamma$.

In the proof that follows, we will never use   the precise formula \eqref{fishtale} for the Poisson kernel, 
but will instead just make use  of the following qualitative properties of $P_{(x,y,z)}(\xi,\eta)$:
\begin{enumerate}[(I)]
 \item \label{uno}
 For any $ \lambda>0$,
       $$P_{(0,0,\lambda)}(\xi,\eta) = \lambda^{-2} 
           P_{(0,0,1)}(\lambda^{-1}  \xi, \lambda^{-1}  \eta ),$$ 
      so that the Poisson kernel scales under dilations 
     to preserve its mass;
       
  \item  
  \label{dos}On any compact 
disk $D_\varrho = \{ \xi^2+\eta^2 \leq \varrho^2\}$, the Poisson kernel $P_{(0,0,1)}(\xi,\eta)$ is bigger than a positive  constant, depending only on $\varrho$; 

  \item  \label{tres}
  $P_{(0,0,1)}(\xi,\eta) \leq 1$ on the whole plane; and 

  \item  \label{quatro}
  $\int_{\RR^2 \setminus D_\varrho} P_{(0,0,1)}(\xi,\eta) \thinspace d\xi \wedge d\eta \longrightarrow0$ 
   as $\varrho \to  \infty$.
\end{enumerate} 
These facts can all be  read off immediately from \eqref{fishtale}, 
although some readers might instead prefer to deduce them  from 
the Poisson kernel's geometric interpretation in terms of  visual area measure. 

Given a quasi-Fuchsian group $\varGamma$ of Bers type, 
we will now choose our upper-half-space model of $\mathcal{H}^3$
so that the point at infinity belongs to $\Omega_-= \Omega_-(\varGamma)$. Thus, 
the limit set $\Lambda = \Lambda (\varGamma)$ will always be presented as a Jordan curve in $\RR^2=\CC$, 
and $\Omega_+= \Omega_+(\varGamma)$ will always be the bounded component of the complement of $\Lambda\subset \CC$. 
Given a measurable subset $E$ of the plane, we will  let $f_E$ denote the function $f$ produced by \eqref{fishdish}
when $F=\chi_E$ is the indicator function of $E$. In particular, $f_{\Omega_+}$ is exactly the pull-back $\widetilde{\zap f}:\mathcal{H}^3 \to (0,1)$
of the tunnel-vision function ${\zap f}$ of $X=\mathcal{H}^3/\varGamma$.  

\begin{thm} \label{criticalpoint} There is an integer $N$  such that, for every closed oriented surface
$\Sigma$ of even genus ${\zap g}\geq N$, there is a 
 quasi-Fuchsian group $\varGamma\cong \pi_1(\Sigma )$  for which 
the corresponding tunnel-vision function ${\zap f}$ has at least two critical points. 
\end{thm} 

\begin{proof} 
Given $\epsilon >0$, set $\varepsilon= \epsilon^3$,  and let 
 $\varGamma_\varepsilon$ be any quasi-Fuchsian group  of Bers type whose limit set $\Lambda_\varepsilon$
is invariant under $\zeta\mapsto -\zeta$ and lies within Hausdorff distance $\varepsilon$ of  the boundary 
of the  ``dogbone'' domain
$$ \Omega_{\epsilon} = \{z:|z-1|<  \frac{1}{4}\} 
                  \cup\{z:  |z+1|< \frac{1}{4}  \}
                  \cup\{z: |z| < 1, |\text{Im}(z)|< \epsilon^3\}.
$$
illustrated by Figure \ref{dogbone1}. By Lemma \ref{closer}, such $\varGamma_\varepsilon$ 
exist for all even genera ${\zap g} \geq N$ for some $N$ depending on $\varepsilon$, and hence on $\epsilon$, but 
 we will not choose a specific $\epsilon$
until much later, so our notation   will help remind us of  this point. 
\bigskip

\begin{figure}[htb]
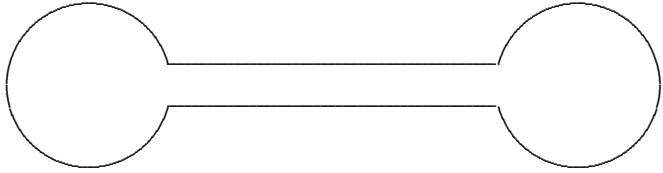

\centerline{
\beginpicture
\setplotarea x from 0 to 320, y from 0 to 60
\circulararc  330 degrees from 100 38 center at 70 30
\circulararc  -330 degrees from 225 38 center at 255 30
{\setlinear 
\plot  100 38 225 38  /
\plot  100 22 225 22  /
}\endpicture
}
\caption{ \label{dogbone1}
A ``dogbone'' domain. The harmonic measure of the interior defines
a harmonic function in the upper half-space that has at least two 
critical points. This phenomenon persists when the boundary curve is approximated
by 
 quasi-Fuchsian limit sets, using   
Lemma \ref{closer}.
}
\end{figure}

\noindent 
As before, we work in  the upper-half-space model $\RR^2 \times \RR^+$ of $\mathcal{H}^3$, 
with coordinates $(x,y,z)$, and follow the convention  that $\Omega_{+,\varepsilon}:= \Omega_+(\varGamma_\varepsilon)$ 
is always taken to be the region {\em inside} the Jordan curve $\Lambda_{\varepsilon}= \Lambda(\varGamma_\varepsilon)$. 
By Lemma \ref{closer}, the limit set $\Lambda_{\varepsilon}$ can always be chosen to be 
symmetric with respect to reflection though the origin in the  $\zeta=\xi+i\eta$ plane, 
and this symmetry then implies that the corresponding harmonic function $f_\epsilon := f_{\Omega_{+,\varepsilon}}$
 is then invariant under the isometry of $\mathcal{H}^3$ represented  in our upper-half-space model 
by reflection $(x,y,z)\mapsto (-x,-y,z)$ through the $z$-axis. Since this means that the gradient $\nabla f_\epsilon = h^{-1}(df_\epsilon, \cdot)$
of  $f_\epsilon$ with respect to the hyperbolic metic $h$ is also invariant under this isometry, it follows that, along the hyperbolic geodesic represented by 
the positive $z$-axis,
the gradient $\nabla f_\epsilon$ must be everywhere 
tangent to the axis. Thus, to show that $f_\epsilon$ has a critical point on the positive $z$-axis, 
it   suffices to show that its restriction to the  $z$-axis is neither monotonically increasing nor decreasing. 
However,    the maximum principle guarantees that  $0 < f_\epsilon < 1$ 
on all of $\uhs$, and  we also know that 
\begin{eqnarray*}
 \lim_{z \searrow 0} f_\epsilon(0,0,z) &=&1,\\
  \lim_{z \nearrow \infty} f_\epsilon(0,0,z) &=&0,
\end{eqnarray*}
since by construction $0 \in \Omega_{+,\varepsilon}$ and $\infty\in \Omega_{-,\varepsilon}$. 
We thus just need  to show that $f_\epsilon$ is not monotonically decreasing along the positive $z$-axis. 

We will do this by showing that 
\begin{equation}
\label{punchline}
f_\epsilon(0,0,\epsilon)  <  f_\epsilon(0,0,1)
\end{equation}
whenever $\epsilon$ is sufficiently small. To see this, first observe that property  \eqref{dos} of the Poisson kernel implies that 
$f_\epsilon(0,0,1)$ is bigger than a positive constant
independent of $\epsilon$, since, for $\epsilon < 1/2$, the region $\Omega_{+,\varepsilon}$ is contained in the disk of radius $3/2$   
about $\zeta=0$, and contains a pair of disks of Euclidean radius $1/8$. On the other hand,  if we  dilate the  upper half-space 
 by a factor of $1/\epsilon$, thereby mapping $(0,0,\epsilon)$  to $(0,0,1)$,  property \eqref{uno}  tells us that we can 
 calculate $f_\epsilon (0,0,\epsilon)$ by instead calculating the Poisson integral at $(0,0,1)$ while replacing $\Omega_{+,\varepsilon}$
 with its dilated image. However, 
 the dilated copy of $\Omega_{+,\varepsilon}$  meets the  disk of radius $1/\epsilon$  in a region of Euclidean area $< 4\epsilon$, 
 so property \eqref{tres} guarantees that 
 the contribution of  this large disk to the integral is $O(\epsilon)$; meanwhile,  property \eqref{quatro} guarantees that the
 the contribution of the exterior of this increasingly large  disk tends to zero as $\epsilon\searrow 0$.  
This establishes that \eqref{punchline} holds for all small $\epsilon$, since $f_\epsilon(0,0,1)$ is bounded away from zero. 
Moreover, this argument shows that this holds for a specific small $\epsilon$ that is 
  independent of the detailed geometry  of the approximating limit sets $\Lambda_\varepsilon$,
just as long as  they are within Hausdorff distance $\varepsilon= \epsilon^3$ of the  ``dogbone'' curve specified by the parameter $\epsilon$.
Thus, there is some specific small $\epsilon$ such that the restriction of the pulled-back tunnel-vision function $\widetilde{\zap f}= f_\epsilon$
associated with any allowed $\Lambda_{\varepsilon}$, where $\varepsilon= \epsilon^3$,  
has the property that its restriction to the $z$-axis has at least two critical points -- namely, a local minimum and a local maximum. Indeed, we 
can even arrange for $\widetilde{\zap f}$ to have different values at these two critical points of the restriction to the axis by insisting that the local maximum 
be the {\em first} local maximum after the local minimum. Since we have also arranged for   the approximating limit sets to be invariant under $\zeta\mapsto -\zeta$,
these critical points of the restriction are actually critical points of $\widetilde{\zap f}$. Since $\widetilde{\zap f}$ assumes different  values at this points, this shows 
that ${\zap f} : X\to \RR$ must have at least
two critical values whenever its limit set $\Lambda_\varepsilon$ satisfies our approximation hypotheses. 
Fixing this $\epsilon$, and applying Lemma \ref{closer} with  $\varepsilon=\epsilon^3$, we thus deduce that there is  
some $N$ such that, for every oriented surface $\Sigma$ of even genus 
${\zap g} \geq N$, there is a quasi-Fuchsian group $\varGamma \cong \pi_1 (\Sigma )$ for which the 
the tunnel-vision function ${\zap f}$ of $X=\uhs/\varGamma$ has  at least two critical points. 
\end{proof} 

Theorem \ref{main1} now follows immediately from Theorem \ref{criticalpoint}, given the fact that any quasi-Fuchsian group is a 
deformation of a Fuchsian one. To prove Theorem \ref{main2}, we merely need to observe that we can simultaneously deform  any quantizable 
configuration  as we deform the relevant Fuchsian group into a given quasi-Fuchsian one. Indeed, suppose that we have a family 
 of quasi-Fuchsian groups smoothly parameterized by the closed interval $[0,1]$. For any specific $t\in [0,1]$,  Lemma \ref{expansion} implies 
 that there is  some $\delta$ such that the tunnel-vision function 
${\zap f} : X\to (0,1)$ does not have any critical values in $(0,\delta)\cup (1-\delta , 1)$, and
we may moreover choose this $\delta$ to be independent of
$t\in [0,1]$ by compactness. By following the gradient flow of $\sqrt{{\zap f}(1-{\zap f})}$ on $\overline{X}$, 
we may therefore construct, for each $t$,  a diffeomorphism between ${\zap f}^{-1}[ (0,\delta)\cup (1-\delta , 1)]$ and $\Sigma \times [ (0,\delta)\cup (1-\delta , 1)]$
such   that ${\zap f}$ becomes projection to the second factor;  moreover,   the diffeomorphism constructed in this way will then  also smoothly  depend on $t$. 
We now assume that $X$ is Fuchsian when $t=0$, so that ${\zap f}$ initially has no critical points. Recall that the quantization condition 
on $\{ p_1, \ldots , p_k\}$ amounts to saying  that  $\sum {\zap f}(p_k) = \ell$
 for some integer $\ell$ with $0< \ell < k$. By first smoothly varying
our configuration within the original Fuchsian $X$, we may then first arrange  that ${\zap f} > 1-\delta$ for $\ell$ of the points, and that ${\zap f}< \delta$ for the rest. 
Once this is done, we may then just consider our configuration as a subset of $\Sigma \times [ (0,\delta)\cup (1-\delta , 1)]$, and our family of diffeomorphisms 
will then  carry $\{ p_1, \ldots , p_k\}$  along as a family of quantizable configuration as we vary $t\in [0,1]$. Theorem \ref{main2} is now an immediate consequence.

\bigskip

Let us  now conclude our discussion with a few comments concerning  Theorem \ref{criticalpoint}.
First of all, 
 the restriction to very large genus ${\zap g}$  appears to be an inevitable limitation  of our method of proof,
because for any fixed genus the possible limit sets $\Lambda (\varGamma)$ only depend on a finite number of parameters, 
and so cannot provide arbitrarily good approximations of a freely specified  curve. On the other hand, the requirement that  
the genus ${\zap g}$ be even is merely a technical convenience. For example, after replacing our dogbone regions with analogous 
domains that are invariant under $\ZZ_p$ for some prime $p$ other than $2$, a similar argument then shows 
that the same phenomenon occurs for all genera ${\zap g}$ that are sufficiently large multiples of $p$. 

In fact, there is a more robust version of our  strategy that should lead to a  proof of  the existence of critical points
without imposing a symmetry condition. Given an $\Omega_+$ that approximates a   dogbone with an extremely narrow corridor between the 
two disks, we would like to understand  the level sets
$S_t$ of the corresponding $\widetilde{\zap f}: \uhs\to (0,1)$. When $t$ is a regular value close to $1$, 
 this surface hugs the boundary of hyperbolic 
space  and is a 
close approximation to the bounded region $\Omega_+$. Now, if $\Omega_+$ were exactly a limiting dogbone consisting of 
 two disks joined by a ``corridor'' of  width zero, one could write down the surfaces in closed form; 
when $t$ is close to $1$,  the surfaces $S_t$ then consists of two ``bubbles'' over the disks, 
but as $t$ decreased these two bubbles  eventually touch and then merge by adding a handle
 joining the two original components. We expect this change of topology to  survive
 small perturbations of the regions in question,  so that one should  be able to prove the existence of a
 critical point by a careful Morse-theoretic argument, just assuming that $\Omega_+$ closely approximates a dogbone
 with a very narrow corridor between the disks. 
A careful implementation 
 of this argument would then enable us to remove the even-genus restriction of Theorem \ref{criticalpoint}. 
 However, we will not attempt to supply all the details in this paper. 
 
 Another issue we have not addressed here is whether  some of our critical points of the tunnel-vision function 
are    non-degenerate in the sense of Morse. This is intrinsically 
 interesting  from the standpoint  of $4$-manifolds, because, as long as  the critical point does not belong
 to the quantizable configuration $\{ p_1, \ldots , p_k\}$,  it 
 is equivalent  to asking whether the corresponding harmonic $2$-form is transverse to the 
 zero section. When such  transversality occurs, it  is stable  is  under small perturbations of the metric,
 and the failure of the metric to be conformally almost-K\"ahler  then persists under small 
 deformations of the conformal class.  
As a matter of fact, we have actually discovered a way to 
 arrange for at least one of our critical points to be  non-degenerate in the sense of Morse. Indeed, if the corridor of our dogbone is so narrow as to 
 be negligeable, direct computation in the spirt of the last paragraph shows that one obtains a critical point 
 at which the Hessian is non-degenerate, and, as long the corridor remains sufficiently thin,  one can therefore show that   the ``second'' critical point of
 our main argument will remain in the region where the determinant of the Hessian is non-zero. Approximation of the relevant dogbone curve 
 by limit sets then leads to quasi-Fuchsian $3$-manifolds for which the tunnel vision function has at least one
 non-degenerate critical point. Precise details will appear elsewhere.

 Finally, while we have proved the existence of two critical points in certain specific 
 situations, it seems natural to expect that there might be many more when the
conformal structures on the  two boundary components of $\overline{X}/\varGamma$ become
widely separated in  Teichm{\"u}ller space. When this happens, the limit set  becomes very chaotic,
and a better understanding of the tunnel-vision function in this setting might eventually   allow one  
to prove the existence of critical points even when the genus ${\zap g}$ is small. 
Conversely, it would be interesting to characterize those quasi-Fuchsian manifolds for
which the tunnel-vision function has relatively few critical points; in particular, 
it would obviously be of great interest to have a precise characterization of those $\varGamma$ for which the
tunnel-vision function has   no critical points at all.

Indeed,  while  this article has used Theorem \ref{pika}  to construct  anti-self-dual deformations of scalar-flat K\"ahler manifolds that are not conformally almost-K\"ahler, 
 one could  in principle use this same result  to instead construct specific examples that {\em are} conformally almost-K\"ahler.
For example, it would be interesting to thoroughly  understand the tunnel-vision functions 
 of the {\em nearly Fuchsian} $3$-manifolds $X$ first studied by Uhlenbeck \cite{karen1}. By definition, these are the quasi-Fuchsian $3$-manifolds of Bers type that contain  
a   compact minimal surface
$\Sigma\subset X$ of Gauss curvature $> -2$. Is  the tunnel-vision function of  a nearly Fuchsian manifold always free of  critical points? If so, 
 then Theorem \ref{pika} would imply that the anti-self-dual $4$-manifolds arising from these special 
quasi-Fuchsian manifolds are always conformally almost-K\"ahler.

In these pages, we  have just scratched the surface of a fascinating subject with 
deep connections to many natural questions in geometry and topology. We can only 
hope that some interested reader will take up the challenge, and address a few 
of the questions we have left unanswered here. 

\pagebreak
%

\end{document}